\documentclass[11pt,a4paper]{article}
\usepackage{amsmath,amsthm,amssymb,url}
\usepackage{graphicx}
\usepackage{fix-cm}%
\usepackage{caption}

\usepackage{tikz, color}%
\usepackage{graphics,epstopdf}%
\usepackage[utf8]{inputenc}
\usepackage[english]{babel}
\usepackage{amsmath}
\usepackage{amsthm}
\usepackage{amssymb}
\usepackage{geometry}
\usepackage{a4wide}
\usepackage{bm}
\usepackage{microtype}
\usepackage{mathtools}
\usepackage{csquotes}
\usepackage[shortlabels]{enumitem}
\usepackage[nottoc]{tocbibind}%
\usepackage{changepage}
\usepackage{cleveref}
\usepackage{enumitem}
\usepackage{xcolor}
\setlist[enumerate]{itemsep=0mm,parsep=2mm}
\usepackage{bbm}

\newtheorem{theorem}{Theorem}[section]
\newtheorem{proposition}[theorem]{Proposition}
\newtheorem{lemma}[theorem]{Lemma}
\newtheorem{claim}[theorem]{Claim}

\newtheorem{cor}[theorem]{Corollary}

\newtheorem{conj}[theorem]{Conjecture}

\theoremstyle{definition}
\newtheorem{question}{Question}

\newcommand{\R}{{\mathbb R}}

\newcommand{\cR}{{\mathcal R}}

\DeclareMathOperator{\rc}{rc}

\newcommand{\mycaption}[1]{
	\refstepcounter{figure}
	\addcontentsline{lof}{figure}{\protect\numberline{\textbf{Fig. \thefigure.}}#1}
	\par\addvspace{\abovecaptionskip}
	\centering\textbf{Fig. \thefigure.} #1
	\par\addvspace{\belowcaptionskip}
}

\usepackage{aliascnt}
\newcounter{romlemma}
\newaliascnt{romlemmaalias}{romlemma}
\aliascntresetthe{romlemmaalias}

\usepackage{thmtools, thm-restate}

\newcommand{\E}{\mathbb{E}}
\newcommand{\Prob}{\mathbb{P}}

\newcommand{\rup}{r ^+}
\newcommand{\dof}{\mathrm{dof}}

\title{\bf Degree Sum Conditions for Graph Rigidity}
\author{Tibor Jord\'an\footnote{Department of Operations Research, ELTE E\"otv\"os Lor\'and University, and HUN-REN--ELTE Egerv\'ary Research Group
		on Combinatorial Optimization, P\'azm\'any P\'eter s\'et\'any 1/C, 1117 Budapest, Hungary.
		e-mail: {\tt tibor.jordan@ttk.elte.hu}}  \and Xuemei Liu\footnote{School of Mathematics and Statistics,
		Northwestern Polytechnical University,
		Xi'an 710129, Shaanxi, P.R. China. e-mail: {\tt xuemeiliu@nwpu.edu.cn} } \and Soma Vill\'anyi\footnote{Department of Operations Research, ELTE E\"otv\"os Lor\'and University, and HUN-REN--ELTE Egerv\'ary Research Group
		on Combinatorial Optimization, P\'azm\'any P\'eter s\'et\'any 1/C, 1117 Budapest, Hungary.
		e-mail: {\tt soma.villanyi@ttk.elte.hu}}}

\date{October 23, 2025}

\begin{document}
	
\maketitle

	\begin{abstract}
		We  study sufficient conditions for the generic rigidity 
		of a graph $G$ expressed in terms of (i) its minimum degree $\delta(G)$, or (ii)
		the parameter
		$\eta(G)=\min_{uv\notin E}(\deg(u)+\deg(v))$.
		For each case, we seek the smallest integers $f(n,d)$ (resp.\ $g(n,d)$) such that every $n$-vertex graph $G$ with $\delta(G)\geq f(n,d)$ (resp.\ $\eta(G)\geq g(n,d)$) is rigid in $\R^d$.
		
		  Krivelevich, Lew, and
		Michaeli conjectured that there is a constant $K>0$ such that $f(n,d)\leq \frac{n}{2}+Kd$ for all pairs $n,d$.
		We give an affirmative answer to this conjecture
        by proving that $K=1$ suffices.
		For $n\geq 29d$, we obtain the exact result	$f(n,d)=\lceil\frac{n+d-2}{2} \rceil$.
		
		Next, we prove that $g(n,d)\leq n+3d$ for all pairs $n,d$,
		and establish
		$g(n,d)=n+d-2$ 
		when
		$n\geq d(d+2)$.
		For $d=2,3,$ we determine the exact values of $f(n,d)$ and $g(n,d)$ for all $n$, confirming another conjecture of Krivelevich, Lew, and
		Michaeli in these low-dimensional special cases.
		
		As an application, we prove that the Erdős-Rényi random graph $G(n,1/2)$ is a.a.s.\ rigid in $\R^d$ for $d=d(n)\sim \frac{7}{32} n$. This result provides the first linear lower bound for $d(n)$, and it answers a question of Peled and Peleg.
	\end{abstract}

	\section{Introduction}

	A $d$-dimensional (bar-and-joint) {\it framework} $(G,p)$ consists of a simple graph $G=(V,E)$ and
	a map $p:V\to \R^d$.
	Such a framework is called {\it rigid} if there is no
	continuous motion $\hat p:[0,1] \to (\R^d)^V$ of the vertices with $\hat p(0)=p$ such that the function $||\hat p_u(t)-\hat p_v(t)||$ is constant for every edge $uv\in E$, but not for all pairs $u,v\in V$. 
	We consider {\it generic} frameworks, in which the
	set of coordinates of $p$ is algebraically independent over $\mathbb{Q}$.
	It is known that, for every fixed $d\geq 1$, the rigidity of a generic framework $(G,p)$ is determined by the graph $G$. We therefore call $G$ {\it $d$-rigid} if every (equivalently, some) generic $d$-dimensional framework $(G,p)$
	 is rigid.

	The first of the two main questions we study in this paper is as follows:
	\begin{question}
		For $n,d\in \mathbb{N}$, what is the smallest integer $f(n,d)$ such that every $n$-vertex graph $G$ with minimum degree $\delta(G)\geq f(n,d)$ is $d$-rigid?
	\end{question}
	
	This problem and its variants were first studied for fixed dimensions $d=2,3$, in the context of global rigidity \cite{B,JJchapter} and matrix completion problems \cite{JJT}.
	Our work is primarily motivated by recent papers of Krivelevich, Lew, and Michaeli \cite{KLM1,KLM}, where the authors
	obtained various bounds on $f(n,d)$, for all $d\geq 1$.
	We first review these developments and outline our contribution.

	For $n\leq d+1$ only complete graphs are $d$-rigid, implying $f(n,d)= n-1$. 
    So we shall frequently assume $n\geq d+2$. 
    It is a well-known fact that $\delta(G) \geq \lceil\frac{n+d-2}{2}\rceil$ is a sufficient condition for $d$-connectivity, and that this bound is best possible. 
	Since $d$-rigid graphs are $d$-connected, we obtain a straightforward yet useful lower bound: \begin{equation}\label{bound:d-2:2}
		\left\lceil\frac{n+d-2}{2}\right\rceil \leq f(n,d).
	\end{equation}
	It follows similarly that $f(n,d)\leq \left\lceil(n+d(d+1)-2)/2\right\rceil$ from the fact that every $d(d+1)$-connected graph is $d$-rigid \cite{vill}.
	More elaborate upper bounds %
	 were established by Krivelevich, Lew, and Michaeli \cite{KLM1}. They proved that
	$f(n,d)\leq (n+3d\log n)/2$ for all pairs $n,d$,
	and showed that this bound can be improved if $n$ is sufficiently large:
	for  $n\geq \Omega(d^2)\cdot \log^2 n$, we have $f(n,d)\leq \frac{n}{2}+d-1$. (See \cite[Theorems 1.14, 1.15]{KLM1}.)
	These results suggest that the function $f(n,d)-\frac{n}{2}$ may depend rather weakly on $n$. In fact, the authors of \cite{KLM1} conjectured that it is bounded above by a linear function of $d$:
	\begin{conj}\cite[Conjecture 1.16]{KLM1}\label{conj:KLM1}
		There exists a constant $K>0$ such that $f(n,d)\leq \frac{n}{2}+Kd$ for all $n,d\in \mathbb{N}$.  
	\end{conj}
	
	Subsequently, in \cite{KLM}, they made further progress towards Conjecture \ref{conj:KLM1} by showing that $f(n,d)\leq \frac{n}{2}+d-1$ already holds for $ n\geq \Omega(d)\cdot\log^2 n$ (\cite[Theorem 1.2]{KLM}).
	Our first result states that here the assumption on $n$ can be dropped, which confirms Conjecture \ref{conj:KLM1} with $K=1$:
	
	\begin{theorem}\label{thm:delta:n/2+d}
		$f(n,d)\leq \frac{n}{2}+d-1$ for all $n,d\in \mathbb{N}$.
	\end{theorem}
	
	Their other main result established that if $n\geq d(2d+1)+2$, then $f(n,d)=\lceil\frac{n+d-2}{2}\rceil$; i.e., the bound in \eqref{bound:d-2:2} is tight (\cite[Theorem 1.1]{KLM}). We improve on this by showing that the assumption on $n$ can be relaxed to $n\geq 29d$.
	
	\begin{theorem}\label{thm:delta:(n+d)/2-1}
		If $n\geq 29d$, then $f(n,d)=\lceil\frac{n+d-2}{2}\rceil.$
	\end{theorem}
	
	For $d< n< 29d$, the exact value of $f(n,d)$ remains an open problem.
	As it was noted in \cite{KLM}, the lower bound in \eqref{bound:d-2:2} is not always tight. A different lower bound arises from the fact that every $d$-rigid graph on more than $d$ vertices has at least $dn-d(d+1)/2$ edges. This yields 
	\begin{equation}\label{bound:bonyi}
		2d-\frac{d(d+1)-1}{n}\leq f(n,d),
	\end{equation} which is more restrictive than $\eqref{bound:d-2:2}$ in the regime $d< n < 2d+2$. 
	Krivelevich, Lew and Michaeli conjectured that the combination of \eqref{bound:d-2:2} and \eqref{bound:bonyi}  is essentially tight:
	\begin{conj} \cite[Conjecture 1]{KLM}
		\label{KLMconj}
		Let $1\leq d<n$. Then
		$$ f(n,d)\leq \max \left\{ \left\lceil\frac{n+d-2}{2}\right\rceil,\left\lceil 2d-\frac{d(d+1)}{n}\right\rceil \right\}.$$
	\end{conj}

	In the second part of this paper, we verify Conjecture \ref{KLMconj} in the special cases $d\in \{2,3\}$. 
	We will, however, consider the problem in a more general form,
	which leads to the following definition and question.
	For a graph $G=(V,E)$, define $$\eta(G)=\min_{uv\notin E}( \deg_G(u)+\deg_G(v)).$$
	
	\begin{question} 
		For $n,d\in \mathbb{N}$, what is the smallest integer $g(n,d)$ such that every $n$-vertex graph $G$ with $\eta(G)\geq g(n,d)$ is $d$-rigid?
	\end{question}
	Since $\eta(G)\geq 2\delta(G)$ for every graph $G$, it follows that  
	$	 f(n,d)\leq \left\lceil g(n,d)/2\right\rceil.$
	We believe that $f(n,d)= \left\lceil g(n,d)/2\right\rceil$ may in fact be true for every $n,d\in \mathbb{N}$.
	We propose this in the form of Conjecture \ref{gnd:conj2} below for the case $n>2d+2$.
	Note that the lower bound from \eqref{bound:d-2:2} corresponds to $g(n,d)\geq n+d-2$.
	In the regime $n> 2d+2$, this dominates the bound one can get from an edge-counting argument (see Lemma \ref{ecount}, for details);
	therefore, we conjecture: 
	\begin{conj}\label{gnd:conj2} Let $n,d \in \mathbb{N}$. %
		If $n> 2d+2$, then 
		$g(n,d) = n+d-2$.
	\end{conj}

	In support of Conjecture \ref{gnd:conj2}, we prove the following theorems.
	
	\begin{theorem}\label{thm:intro:g:5d}
		$g(n,d)\leq n+3d-3$ for all $n,d\in \mathbb{N}$.
	\end{theorem} 
	
	\begin{theorem}\label{thm:intro:g:d-1}
		If $n\geq d(d+2)$, then $g(n,d)= n+d-2$.
	\end{theorem} 
	
	Theorem \ref{thm:intro:g:d-1} is a strengthening of \cite[Theorem 1.1]{KLM} mentioned above, and it implies Conjecture \ref{gnd:conj2} in the regime $n\geq d(d+2)$.
	Next, we verify Conjecture \ref{gnd:conj2} in the special cases $d\in \{2,3\}$.		
	Surprisingly, the major difficulties still arise when $n$ is particularly small.
	In Theorem
	\ref{rigid graph in R^3}, we show that
	if
	$\eta(G)\geq n+1$, then
	either 
	$G$ is isomorphic to
	one of four special graphs on at most seven vertices, or $G$ is rigid in $\R^3$.
	We obtain a similar result for $d=2$, as well (Theorem \ref{rigid graphs in R2}).
	These results imply an
	affirmative answer to Conjecture \ref{KLMconj} for
	$d=2,3$ (Theorems \ref{conjd3}, \ref{conjd2}).

	We briefly address the corresponding questions on global rigidity, and show that the (optimal) sufficient conditions for $d$-rigidity in terms of $\delta(G)$ or $\eta(G)$ directly imply (optimal) sufficient conditions for global $(d-1)$-rigidity.
	
	Finally, we apply our results to settle a question raised by Peled and Peleg \cite{PP}, asking whether there exists $\varepsilon>0$ such that the Erdős-Rényi random graph $G=G(n,1/2)$ is asymptotically almost surely $(\varepsilon n)$-rigid. For a brief outline of the background to this question, see Section \ref{sec:appl:random}. We will prove the following:
	\begin{restatable}{theorem}{random}\label{thm:random:1/2}
	   The random graph $G\sim G(n,1/2)$ is a.a.s.\ $d$-rigid for $d=\frac{7}{32}n-\frac{\sqrt{15n\log n}}{16}$.
	\end{restatable}

	Our proof methods are substantially different from those of \cite{KLM1}
	and \cite{KLM}. We verify new structural properties of the
	$d$-dimensional rigidity matroid and in most of the proofs we use
	 matroid theoretic arguments. In some cases we also
	rely on versions of a probabilistic argument that was
	introduced in \cite{vill} and further developed in \cite{JJV}.

	\subsubsection*{Paper organization and notation}
	
	In Section \ref{sec:prelim}, we collect the basic properties of
	$d$-rigid graphs and the $d$-dimensional generic rigidity matroid. 
	 In Section \ref{sec:mindegree}, we derive the sufficient conditions for $d$-rigidity that are expressed in terms of the minimum degree of a graph. 
	 Section \ref{sec:degreesum} contains the results concerning the more general bounds, based on degree sum conditions, for higher dimensions. 
     Section \ref{sec:lowdim} is dedicated to the cases $d=2,3$.
	In  Section \ref{sec:appl:random}, we show how our results can be applied to random graphs. 
	We conclude the paper by deducing similar sufficient conditions for
	global rigidity in Section \ref{sec:conluding}. 
    \medskip

	Throughout the paper we use the terms rigidity in $\R^d$ and $d$-rigidity interchangeably. We only consider simple graphs, that is, graphs without loops and parallel edges. 
	For a graph $G$, we let $V(G)$ and $E(G)$ denote the vertex and edge sets of $G$, respectively. For a subset $X \subseteq V(G)$ we let $G[X]$ denote the subgraph of $G$ induced by $X$, and we let $i_G(X) = |E(G[X])|$ be the number of edges induced by $X$ in $G$. We use $K(X)$ to denote the complete graph on vertex set $X$, and $K_n$ denotes the complete graph on $n$ vertices.  Let $C_n$ be the cycle on $n$ vertices. Given a positive integer $k$, we say that a graph is \emph{$k$-connected} if it has at least $k+1$ vertices and it remains connected after the removal of any set of fewer than $k$ vertices.
	In a graph $G=(V,E)$ we use $e_G(X,Y)$ to denote the number of
	edges from $X-Y$ to $Y-X$ for $X,Y\subseteq V$. When $Y=V-X$, we use $e_G(X)$ instead of $e_G(X,Y)$. For $X=\{v\}$ we simply write $e_G(v,Y)$,
	and we put $\deg_G(v)=e_G(v,V-\{v\})$.
	Given a vertex $v \in V(G)$, $N_G(v)$ denotes the set of neighbors of $v$.
	Thus $\deg_G(v) = |N_G(v)|$ is the degree of $v$.
	The minimum degree in $G$ is denoted by $\delta(G)$.
	For a vertex $v$ in $V$ and $Y\subseteq V$, define $N_G(v,Y)=N_G(v)\cap Y$.
	We omit the graph $G$ from the subscripts if $G$ is clear from the context.

	\color{black}

	\section{Preliminaries}
	\label{sec:prelim}

    \begin{lemma}
			\label{4conn}
			Let $G=(V,E)$ be a non-complete graph on $n$ vertices and let $r$ be an integer with $r\leq n-3$.
			Suppose that \begin{equation}\label{eq:n+r}
			    \deg(u)+\deg(v)\geq n+r
			\end{equation} for all
			pairs $u,v\in V$ of non-adjacent vertices.
			Then $G$ is $(r+2)$-connected. Furhermore, the bound \eqref{eq:n+r} is optimal.
		\end{lemma}
		
		\begin{proof}
			The lower bound on the degree sum implies that
			$u$ and $v$ have at least $r+2$
			common neighbours in $G$ for each non-adjacent vertex pair $u,v$.
			Thus $G$ cannot have a vertex separator of size at most $r+1$. Furthermore, graphs constructed by gluing together two disjoint complete graphs along $r+1$ vertices demonstrate that the bound \eqref{eq:n+r} is optimal.
		\end{proof}

	\subsection{Graph operations that preserve rigidity}

	Let $G =(V,E)$ be a graph. The $d$-dimensional $0$-$extension$ operation adds a new vertex $v$ to $G$ and $d$ new edges incident with $v$. The $1$-$extension$ operation removes an edge $uw$, and adds a new vertex $v$ and $d+1$ new edges, including $vu$, $vw$.
	The following lemma is well-known, for a proof, see e.g.\ \cite{Whlong}.
	
	\begin{lemma}
		\label{ext}
		Let $H$ be a rigid graph in $\R^d$ and let $G$ be obtained from $H$ by a $0$- or
		$1$-extension. Then $G$ is rigid in $\R^d$.
	\end{lemma}

	Suppose $z$ is a vertex of a graph $H$. 
	The {\em $d$-dimensional vertex splitting (resp.\ spider splitting) operation} constructs a new graph $G$ from $H$  by deleting $z$ and then adding two new vertices $u$ and $v$ with $N_G(u)\cup N_G(v)=N_H(z)\cup \{u,v\}$ 
	(resp. $N_G(u)\cup N_G(v)=N_H(z)$)
	and $|N_G(u)\cap N_G(v)|\geq d-1$ (resp. $|N_G(u)\cap N_G(v)|\geq d$).
	We can view vertex splitting as an inverse operation to contracting the edge $uv$, while
	spider splitting corresponds to contracting a non-adjacent vertex pair.
	Whiteley \cite{Wvsplit} showed that 
	these operations preserve rigidity in $\R^d$, see also \cite{Whlong}.
	We shall refer to this result as Whiteley's vertex splitting theorem.

	The following lemmas are
	easy to prove by using the 3-dimensional version of Lemma \ref{ext}.
	
	\begin{lemma}
		\label{cycle}
		Let $H = (V,E)$ be a rigid graph in $\R^3$ and
		let $\{v_i^1,v_i^2\}$, $1\leq i\leq k$, be pairs of distinct vertices of $H$
		with
		$|\cup_{i=1}^k \{v_i^1,v_i^2\}| \geq 3$ and $k\geq 3$. Let $G$ be obtained from the disjoint union of $H$ and a $k$-cycle $u_1u_2\cdots u_ku_1$ by adding all edges $u_iv_i^1,u_iv_i^2$,
		for $1 \leq i\leq k$. Then
		$G$ is rigid in $\R^3$. 
	\end{lemma}
	
	\begin{lemma}
		\label{union}
		Suppose that a graph $G=(V,E)$ contains two disjoint subgraphs $H_1,H_2$ on at least three vertices with
		$V=V(H_1)\cup V(H_2)$
		such that $H_1,H_2$ are both rigid in $\R^3$ and  $e(v,V(H_1))\leq 2$ for all $v\in V(H_2)$. 
		If
		$H_1$ contains a set $A=\{a_1,a_2,a_3\}$ of three vertices with $e(a_i,V(H_2))\geq 2$ for all $1\leq i\leq 3$,
		then $G$ is rigid in $\R^3$.
	\end{lemma}
	
	\subsection{The rigidity matroid}
	
	Let $d$ be a positive integer. 
	The {\it $d$-dimensional generic rigidity matroid} of a graph $G$, denoted by ${\cal R}_d(G)$,
	is defined on the  
	edge set of $G$. It is the restriction of ${\cal R}_d(K_{|V(G)|})$ to the edge set of $G$.
    This matroid encodes several fundamental rigidity properties of 
	(generic $d$-dimensional bar-and-joint frameworks of)
	$G$. 
	For example, the rigidity of $G$ in $\R^d$ can be characterized by the rank of ${\cal R}_d(G)$, for all $d\geq 1$. 
	We shall denote the rank of ${\cal R}_d(G)$ by $r_d(G)$.
	It is well-known that 
	\begin{equation}
		\label{gl}
		r_d(G)\leq d|V|-\binom{d+1}{2} \mbox{ for all graphs $G=(V,E)$  with $|V|\geq d+1$.}\notag
	\end{equation}
	Equality holds if and only if $G$ is rigid in $\R^d$.
	The ($d$-dimensional) {\it degrees of freedom} of $G$, denoted by $\dof_d(G)$, is the
	difference between the rank of a rigid graph on $V(G)$ and $r_d(G)$.
	Thus, if $|V|\geq d+1$, then $\dof_d(G)=d|V|-\binom{d+1}{2}-r_d(G)$.
	We say that a graph $G=(V,E)$ is 
	{\it ${\cal R}_d$-closed}, if $r_d(G+uv)=r_d(G)+1$ for all non-adjacent
	vertex pairs $u,v\in V$.
	A non-adjacent pair $u,v\in V$ is {\it ${\cal R}_d$-linked} (resp.
	{\it ${\cal R}_d$-loose}) in $G$ 
	if $r_d(G+uv)=r_d(G)$ (resp. $r_d(G+uv)=r_d(G)+1$) holds.
	The {\it ${\cal R}_d$-closure} of $G$ is the unique maximal 
	supergraph $G'$ of $G$ with $r_d(G')=r_d(G)$. It is obtained from $G$
	by adding the edges $uv$ for all ${\cal R}_d$-linked pairs $u,v$ of $G$.
	
	We call an edge $e\in E$ an {\itshape ${\cal R}_d$-bridge} of $G$
if $r_d(G-e)=r_d(G)-1$. A graph $H$ with edge set $E$ is 
{\itshape ${\cal R}_d$-independent} (resp. 
an 
{\itshape ${\cal R}_d$-circuit}), 
if $E$ is independent (resp.
$E$ is a circuit) in ${\cal R}_d(H)$. Thus an edge is
an ${\cal R}_d$-bridge of $G$ if and only if it is not contained
by a subgraph $H$ of $G$ which is an ${\cal R}_d$-circuit.

	Let $G=(V,E)$ be a graph and $U\subseteq V$.
	The graph obtained from $G$ by adding a new vertex $w$ and a new
	edge from $w$ to each vertex in $U$ is denoted by $G^{w,U}$.
	The graph $G^w:=G^{w,V}$ is called the 
	{\it cone} of $G$.
	Part (i) of the following theorem is a fundamental result of Whiteley \cite{Whcone}. Parts (ii) and (iii) are simple corollaries of (i).
	
	\begin{theorem} \cite{Whcone}
		\label{thm:coning}
		Let $G=(V,E)$ be a graph and let $u,v\in V$. Then\\
		(i) $G$ is ${\cal R}_d$-independent
  (resp. $d$-rigid) if and only if $G^w$ is ${\cal R}_{d+1}$-independent
  (resp. $(d+1)$-rigid),\\
		(ii) $\{u,v\}$ is $\mathcal{R}_d$-linked in $G$ if and only if $\{u,v\}$ is $\mathcal{R}_{d+1}$-linked in $G^w$,\\
		(iii) $\dof_d(G)=\dof_{d+1}(G^w)$.
	\end{theorem}

	The next lemma is also a corollary of the coning theorem.
	
	\begin{lemma}\label{lem:dim:addedges}
		Let $G=(V,E)$ be a graph and let $\{u,v\}$ be an 
		$\mathcal{R}_d$-loose vertex pair of $G$.
		Suppose that $F$ is a set of new edges with $uv\notin F$. 
		Then $\{u,v\}$ is $\mathcal{R}_{d+|F|}$-loose in $G+F$.
	\end{lemma}
	\begin{proof}
		It suffices to prove the statement for $|F|=1$. Suppose that $F=\{xy\}$.
		We may assume that $x\notin \{u,v\}$.  Let $G'$ be the graph obtained from $G$ 
		by adding the edges $xz$ for all $z\in V-N_G(x)-\{x\}$. Since $\{u,v\}$ is $\mathcal{R}_{d}$-loose in $G$, it is also $\mathcal{R}_{d}$-loose in $G-x$. Hence
		$\{u,v\}$ is $\mathcal{R}_{d+1}$-loose in $G'$ by Theorem \ref{thm:coning}. Since $G+F$ is a subgraph of $G'$, the lemma follows.
	\end{proof}

    	We shall use the next simple lemma several times.
	
	\begin{proposition}\label{prop:easy:bridges}
		Let $G=(V,E)$ be a graph, and let $u,v,w\in V$ be distinct vertices.
		Suppose that $u\in N_G(v)\subseteq N_G(w)\cup\{w\}$ and $uw$ is an $\mathcal{R}_d$-bridge in $G$. Then $uv$ is an $\mathcal{R}_d$-bridge in
		$G-w$.
	\end{proposition}
	\begin{proof}
		Suppose not. Then there exists an ${\cal R}_d$-circuit $C$ in $G-w$
		containing $uv$. Since $N_G(v)\subseteq N_G(w)\cup\{w\}$, there is a
		subgraph $C'$ in $G-v$ which is isomorphic to $C$ and contains $uw$.
		Then $C'$ is an 
		${\cal R}_d$-circuit in $G$ containing $uw$, contradicting the
		assumption that $uw$ is an $\mathcal{R}_d$-bridge in $G$.
	\end{proof}
	
	We refer the reader to \cite{Jmemoirs,SW} for more details concerning
	rigid graphs and the rigidity matroid.
	
	\section{Minimum degree conditions for $d$-rigidity}
	\label{sec:mindegree}
	In this section, we prove Theorems \ref{thm:delta:n/2+d} and \ref{thm:delta:(n+d)/2-1}. 
		To do so, we begin by introducing a new graph parameter, denoted by $\rup_d(G)$, which compares the rank of the cone graph of $G$ to the rank of $G$ in the rigidity matroid. First, we derive some auxiliary statements about $\rup_d(G)$ (Subsection \ref{subsec:rupd}). We then use this parameter to give a purely combinatorial proof of Theorem 
	\ref{thm:delta:n/2+d}
	(Subsection \ref{subsec:proof+d}).
	Next,  we describe the notion of the rank contribution of a vertex (Subsection \ref{subsec:rankcontr}), another key tool in our  arguments, which we apply in the proof of Theorem \ref{thm:delta:(n+d)/2-1} (Subsection \ref{subsec:proof+d/2}).

	\subsection{Properties of $\rup_d(G)$}
	\label{subsec:rupd}

    Let $G=(V,E)$ be a graph. We define
	$$	\rup_d(G) =
	r_d(G^w)-r_d(G).
	$$
	For graphs with $|V|\leq d$, we have $\rup_d(G)=|V|$ by Lemma \ref{ext}. If  $|V|\geq d$, then %
	Theorem \ref{thm:coning}(iii)
	implies that
	\begin{equation}\label{eq:rup:n>d1}
		\rup_d(G)=\dof_d(G)-\dof_d(G^w)+d= \dof_d(G)-\dof_{d-1}(G)+d.
	\end{equation}
	
	\begin{lemma}\label{lem:dof:coning} 
		If $|V|\geq d$, then $\rup_{d+1}(G^w)=\rup_{d}(G)+1$.
	\end{lemma}
	\begin{proof} By using the equalities in  \eqref{eq:rup:n>d1},
		we obtain
		\begin{align*}
			\rup_{d+1}(G^w)-\rup_{d}(G)&=(\dof_{d+1}(G^w)-\dof_{d}(G^w)+d+1)-(\dof_d(G)-\dof_d(G^w)+d)\\
			&=\dof_{d+1}(G^w)-\dof_d(G)+1=1,
		\end{align*}
		where the last equality follows from Theorem \ref{thm:coning}(iii).
	\end{proof}

	\begin{lemma}\label{lem:rup:subgraph} 
		Let $H$ and $G$ be graphs on the same vertex set. If $H$ is a subgraph of $G$, then $\rup_d(H)\geq \rup_d(G)$.
	\end{lemma}
	
	\begin{proof}
		By the submodularity of $r_d$, we have $r_d(H^w)+r_d(G)\geq r_d(G^w)+r_d(H)$, which implies $\rup_d(H)\geq \rup_d(G)$.
	\end{proof}

	The next lemma gives
	a lower bound on $\rup_{d'}(G-v)$ when $G$ is an $\mathcal{R}_d$-closed graph and $d\leq d'$.
	
	\begin{lemma}\label{lem:dofup:s1s2s3}
		Let $d\leq d'$ be positive integers, $G=(V,E)$ an $\mathcal{R}_d$-closed graph, and $u,v\in V$. 
		Define
		$s_1=\min(|V-N_G(v)-\{v\}|,d'-d+1)$,
		$s_2=\min({|N_G(v)-N_G(u)-\{u\}|}, {d'-d})$, and
		$s_3=\min(|N_G(u)\cap N_G(v)|,d')$.
		Then 
		\begin{enumerate}[label=(\alph*)]
			\item $\rup_{d'}(G-v)\geq r_{d'}(G)-r_{d'}(G-v)+s_1,$ and
			\item $\rup_{d'}(G-v)\geq s_1+s_2+s_3.$
		\end{enumerate}
	\end{lemma}
	
	\begin{proof} 
		(a)
		Let $X_1$ denote a subset of  $V-N_G(v)-\{v\}$ of size $s_1$, and let $G^+$ denote the graph obtained from $G$ by adding the edges $xv$ for all $x\in  X_1$. Since $G$ is $\mathcal{R}_d$-closed, $\{x,v\}$ is ${\cal R}_d$-loose in $G$
		for each $x\in X_1$, and hence each new edge $xv$ is an $\mathcal{R}_{d'}$-bridge in $G^+$ by Lemma \ref{lem:dim:addedges}.
		Thus $$r_{d'}(G)+s_1= r_{d'}(G^+) \leq r_{d'}((G-v)^w)= \rup_{d'}(G-v)+r_{d'}(G-v) .$$
		
		(b) By part (a), it is sufficient to prove that $r_{d'}(G)\geq r_{d'}(G-v)+s_2+s_3$.
		Let $X_2$ (resp.\ $X_3$) denote a subset of $N_G(v)-N_G(u)-\{u\}$  (resp.\ $N_G(u)\cap N_G(v)$) of size $s_2$ (resp.\ $s_3$).
		Let $G_0$ (resp.\ $G_1$) denote the graph obtained from $G$ by first deleting all the edges incident with $v$ and then joining $v$ to the vertices of $X_3$ (resp.\ $X_2\cup X_3$). Then $r_{d'}(G_0)-r_{d'}(G-v)=s_3$ by Lemma \ref{ext}.
		Define $F=\{ux:x\in X_2\}$. By Lemma \ref{lem:dim:addedges}, every $ux\in F$ is an $\mathcal{R}_{d'-1}$-bridge in ${G_1}+F$. Note that each neighbour of $v$ is a
		neighbour of $u$ in $G_1+F$.
		Thus, by Proposition \ref{prop:easy:bridges}, for every $x\in X_2$, the edge $vx$ is an $\mathcal{R}_{d'-1}$-bridge in $G_1-u$, and hence an $\mathcal{R}_{d'}$-bridge in $G_1$. Therefore, $r_{d'}(G_1)-r_{d'}(G_0)=s_2$. 
		It follows that $r_{d'}(G)\geq r_{d'}(G_1)= r_{d'}(G_0)+s_2= r_{d'}(G-v)+s_2+s_3,$ as required.
	\end{proof}
	
	To conclude this subsection, we derive an upper bound on $\rup_d(G)$ expressed in terms of the minimum degree of $G$.
	
	\begin{lemma}\label{lem:dof:easybound:v2}
		Let $G=(V,E)$ be a graph on $n\geq d$ vertices with minimum degree $\delta$.
		Then %
		\begin{equation}\label{eq:dofup:delta}
			\rup_d(G)(\delta -d+2)\leq d(n-d+1).
		\end{equation}
		In particular, if $\delta\geq \frac{n+d-2}{2}$, then $\rup_d(G)< 2d$.
	\end{lemma}
	\begin{proof}
		It follows from the definition of $\rup_d$ that
		there is a set $U\subseteq V$ of vertices such that $|U|= \rup_d(G)\geq d$ and $$r_d(G^{w,U})=r_d(G)+|U|.$$ 
		Let $H=(V,F)$ be the subgraph of $G$ defined by $F=\big\{uv\in E: \{u,v\}\cap U\neq \emptyset\big\}$.
		For each $u\in U$, the edge $uw$ is an $\mathcal{R}_d$-bridge in $G^{w,U}$, and hence also in its subgraph $(H[U])^w$. 
		Thus, by Theorem \ref{thm:coning}, the graph $H[U]$ is $\mathcal{R}_{d-1}$-independent. 
		This implies that
		$|F(U)|\leq (d-1)|U|-\binom{d}{2}.$ 
		Moreover, by Proposition \ref{prop:easy:bridges}, for each $v\in V-U$, every edge $e\in F$ incident with $v$ is an $\mathcal{R}_d$-bridge in $H$. 
		Since $H$ is obtained from $H[U]$ by adding new vertices and $\mathcal{R}_d$-bridges, it follows that $H$ is  $\mathcal{R}_d$-independent. 
		Thus, using (\ref{eq:rup:n>d1}) we obtain 
		$$|F|\leq dn-\binom{d+1}{2}-\dof_d(G)\leq d(n+1)-\binom{d+1}{2}-\rup_d(G).$$ 
		Note that $\deg_H(u)=\deg_G(u)$ for all $u\in U$.
		It follows that $$ \delta |U|-\Big((d-1)|U|-\binom{d}{2}\Big)\leq \sum_{u\in U} \deg_H(u)- |F(U)|=|F|< d(n+1)-\binom{d+1}{2}-\rup_d(G).$$
		By rearranging the inequality and using $|U|=\rup_d(G)$, we obtain \eqref{eq:dofup:delta}. 
		If $\delta\geq \frac{n+d-2}{2}$, then $\rup_d(G)< 2d$ follows from \eqref{eq:dofup:delta}.
	\end{proof}

	\subsection{Proof of Theorem \ref{thm:delta:n/2+d}}	\label{subsec:proof+d}

		We are ready to prove Theorem \ref{thm:delta:n/2+d}, which states that every graph $G$ with minimum degree 
	\begin{equation}
		\label{alln}
		\delta(G)\geq \frac{n}{2}+d-1.
	\end{equation}
	is rigid in $\R^d$.
	
	\begin{proof}[Proof of Theorem \ref{thm:delta:n/2+d}]
		The proof is by induction on $d$. The case $d=1$ follows from Lemma \ref{4conn}. 
		Let $d\geq 2$ and suppose, for a contradiction, that there exists a
		counter-example, that is, a 
		non-rigid graph $G=(V,E)$ that satisfies (\ref{alln}).
		We may assume that $|V|$ is as small as possible, and subject to this,
		$|E|$ is as large as possible.
		Since $G$ satisfies (\ref{alln}) and $G$ is not complete, we have
		$n\geq 2d+2$. Furthermore, the maximality of $|E|$ implies that $G$ is $\mathcal{R}_d$-closed.
		
		If $\delta(G)> \frac{n}{2}+d-1$, then $G'=G-v$ also satisfies (\ref{alln}) for
		any $v\in V$.
		Moreover, $G'$ is not rigid by Lemma \ref{ext}, contradicting the choice
		of $G$. Hence $n$ is even, and there exists a vertex
		$u\in V$ with $$\deg_G(u)=\frac{n}{2}+d-1.$$ 
		Let $X=V-N_G(u)-\{u\}$ be the set of vertices not adjacent to $u$.
		We have $|X|=\frac{n}{2}-d$. 
		  Since $\sum_{v\in V}|N_G(v)\cap X|=\sum_{x\in X}\deg_G(x)> \frac{n}{2}|X|$, it follows that there exists a vertex $v\in V$ such that $|N_G(v)\cap X|\geq \frac{|X|+1}{2}$.
		Let
		$U=N_G(u)\cup N_G(v)-\{u,v\}$.
		Then 
		\begin{equation}\label{eq:VminusU}
			|V-U|\leq |X-N_G(v)|+2\leq \frac{|X|+3}{2} =\frac{n}{4} - \frac{d-3}{2}.    
		\end{equation}

		\begin{claim}\label{claim:2dproof}
			$|V-U|\geq d+2.$
		\end{claim}
		\begin{proof}
			Consider the graph $G'$
			obtained from $G$ by contracting the vertex pair $u,v$ into a new vertex $w$.
			Note that $u$ and $v$ have at least $2d-2$ common neighbours by (\ref{alln}).
			Thus, depending on whether $u$ and $v$ are adjacent or not, $G$ can be obtained from $G'$ by a $d$-dimensional vertex splitting or spider splitting
			operation. Therefore $G'$ is not rigid in $\R^d$ by Whiteley's vertex splitting theorem.
			Moreover, we have $\delta(H)\geq \frac{|V(H)|}{2}+d-2$, where
			$H=G'-w$.
			Thus $H$ is rigid in $\R^{d-1}$ by induction.
			By Theorem \ref{thm:coning}(i) $H^w$ is rigid in $\R^{d}$.
			As $H^w$ is obtained from $G'$ by adding edges incident with $w$, it follows that		
			there is a vertex $z\in V-\{u,v\}$ such that $\{z,w\}$ is $\mathcal{R}_d$-loose in $G'$.

			Let $t=|V-U|$ and 
			suppose, for a contradiction, that $t\leq d+1$.
			Then $H^w$ can be obtained from $G'$ by adding $t-2$ new edges, one of which is $zw$. Thus, by Lemma \ref{lem:dim:addedges}, $zw$ is an $\mathcal{R}_{d+t-3}$-bridge in $H^w$. By Proposition  \ref{prop:easy:bridges}, each edge incident with $z$ is an $\mathcal{R}_{d+t-3}$-bridge in $H$. 
			Hence, $ r_{d+t-3}(H)-r_{d+t-3}(H-z)= \deg_H(z)\geq \frac{n}{2}+d-1.$
			 By Lemmas \ref{lem:dof:coning} and \ref{lem:rup:subgraph}, it follows that $$\rup_{d+t-2}(H)\geq \rup_{d+t-2}((H-z)^z)=  \rup_{d+t-3}(H-z)+1\geq \frac{n}{2}+d.$$
			Let $k=\rup_{d+t-2}(H)$.
			By applying Lemma \ref{lem:dof:easybound:v2} to the graph $H$,
			minimum degree $\frac{n}{2}+d-3$ and dimension $d+t-2$,
			we have 
			$$k\Big(\Big(\frac{n}{2}+d-3\Big)-(d+t-2)+2\Big)\leq (d+t-2)\big((n-2)-(d+t-2)+1\big).$$
			Simplifying terms and using $ \frac{n}{2}+d\leq k$, we obtain
			$$\frac{n}{2}+d\leq (d+t-2)\frac{n-d-t+1}{\frac{n}{2}-t+1}\leq 2(d+t-2).$$
			This implies $\frac{n}{2}-(d-4)\leq 2t $, which contradicts \eqref{eq:VminusU}.
		\end{proof}	 
		
		By (\ref{eq:VminusU}) and Claim \ref{claim:2dproof}, we have $\frac{3n}{4}+\frac{d-3}{2}\leq |U|\leq n-d-2$, which implies $6d+2\leq n$. 
		Hence, $$|N_G(v)-N_G(u)-\{u\}|=|N_G(v)\cap X|\geq \frac{|X|+1}{2} = \frac{\frac{n}{2}-d+1}{2}\geq d.$$
		Since $\delta(G)\geq \frac{n}{2}+d-1$ and $N_G(u)\cup N_G(v)\neq V$, we have $|N_G(u)\cap N_G(v)|\geq 2d-1$. From Claim \ref{claim:2dproof} it also follows that $|V-N_G(v)-\{v\}|\geq d$. 
		Thus, applying Lemma \ref{lem:dofup:s1s2s3}(b) with $d'=2d-1$, $s_1=d$, $s_2=d-1$, $s_3=2d-1$ gives $\rup_{d'}(G-v)\geq 4d-2=2d'$. %
		Since $\delta(G-v)\geq  \frac{(n-1)+d'-2}{2}$, this contradicts Lemma \ref{lem:dof:easybound:v2}, which completes the proof.
	\end{proof}

	\subsection{Expected rank contribution of vertices in rigidity matroids}
	\label{subsec:rankcontr}
	
	In the subsequent proofs, we will frequently employ a probabilistic argument introduced in \cite{vill} and further developed in \cite{JJV}. Our exposition follows the approach of \cite{JJV}.

	Let $d$ be a positive integer, $G=(V,E)$ be a graph and $\pi$ be a uniformly random ordering of $V$.
	For a  vertex $v\in V$, let $T_v^\pi$ denote the set of those vertices that precede $v$ in $\pi$.
	We define the {\it rank contribution of $v$ in $\cR_d(G)$}, with respect to $\pi$,  to be $$\rc_d(G, v, \pi) = r_d(G[T_v^\pi\cup \{v\}])-r_d(G[T_v^\pi]),$$
	and the {\it rank contribution of $v$ in $\cR_d(G)$} to be
	
	$$\rc_d(G,v)= \E\big(\rc_d(G,v,\pi)\big).$$
    As a consequence of the linearity of expectation, one obtains the following lemma.
	
	\begin{lemma}
		\cite{JJV}
		\label{lemma:rc_sum} For any graph $G=(V,E)$, we have ${\displaystyle r_d(G)=\sum_{v\in V} \rc_d(G,v)}.$
	\end{lemma}
	
	    	We next introduce a lower bound of $\rc_d(G,v)$.
	Let $E_v$ denote the set of edges of $G$ incident with $v$. Put $N_v^\pi=T_v^\pi\cap N_G(v)$ and
	$E_v^\pi=\{vu\in E_v:u\in N_v^\pi\}$.
	We define 
	$$\rc^*_d(G, v, \pi) = r_d(G-E_v+E_v^\pi)-r_d(G-v)$$
	and 
	$$\rc^*_d(G,v)= \E\big(\rc^*_d(G,v,\pi)\big).$$
	   The submodularity of $r_d$ implies that $\rc^*_d$ is a lower bound for $\rc_d$:
	
	\begin{lemma} \cite{JJV}
		\label{lemma:rc_rcstar} For any graph $G=(V,E)$ and $v\in V$, we have
		$\rc^*_d(G,v)\leq \rc_d(G,v).$ 
	\end{lemma}
	
	   In what follows, we give another equivalent form of $\rc^*_d(G,v)$. Let $\mathcal{R}_d^v(G)$ denote the matroid obtained from $\mathcal{R}_d(G)$ by contracting all edges not incident with $v$; that is,
	$$\mathcal{R}_d^v(G) = \mathcal{R}_d(G) / E(G - v).$$
	The ground set of $\mathcal{R}_d^v(G)$ is $E_v$.
	Let $r_d^v:2^{E_v}\to \mathbb{N}$ denote the rank function of $\mathcal{R}_d^v(G)$. Then $r_d^v(E_v)=r_d(G)-r_d(G-v)$, and $$r_d^v(E_v^\pi)=\rc_d^*(G,v,\pi).$$
    	Set $k=\deg_G(v)=|E_v|$. Then
    for any $i\in\{0,1,\dots, k\}$, we have $$\Prob(|E_v^\pi|=i)=\frac{1}{k+1}.$$
    Observe that conditioned on $|E_v^\pi|=i$, the set $E_v^\pi$ is uniformly distributed among all $\binom{k}{i}$ subsets of $E_v$ of size $i$. Hence,
	
	\begin{equation}\label{eq:rc*:equi:form}
		\rc_d^*(G,v)=\sum_{i=0}^{k} \E\big(r_d^v(E_v^\pi)\mid |E_v^\pi|=i\big)\, \Prob(|E_v^\pi|=i)=\sum_{i=0}^{k} \frac{\E(r_d^v(A_i))}{k+1}, 
	\end{equation}
	where $A_i$ is a subset of $E_v$ of size $i$ chosen uniformly at random for all $0\leq i\leq k$.\medskip
    
     By imposing mild conditions on the local structure of $G$, one can often establish (sufficiently strong) lower bounds for $\rc^*_d(G,v)$, which in turn yield bounds for $\rc_d(G,v)$ and $r_d(G)$. This subsection presents several statements of this kind.
	The {\em girth} of a matroid ${\cal M}$ is the size of the smallest
	circuit of ${\cal M}$ (or $+\infty$, if $\cal M$ contains no circuits).

	\begin{lemma} \label{lem:matroid:randomset}
		Let ${\cal M}=(E,r)$ be a matroid with
		rank function $r$ and with girth at least $d+1$. Let $i$ be an integer with $d\leq i\leq |E|$, and $A$ be a uniformly random subset of $E$ of size $i$. Then $$\E(r(A))\geq d+(r(E)-d)\frac{i-d}{|E|-d}.$$  
	\end{lemma}
	
	\begin{proof}
		For every $X\subseteq E$ with $|X|=d$, let us fix a subset $\hat X$ of size $r(E)-d$ such that $X\cup \hat X$ is a basis of ${\cal M}$.
		Let \( (A_0, A_1) \) be a pair of random subsets of \( E \) defined as follows: first, \( A_0 \) is chosen uniformly at random among all subsets of \( E \) of size \( d \); then, conditioned on \( A_0 \), the set \( A_1 \) is chosen uniformly at random among all subsets of \( E \setminus A_0 \) of size \( i - d \).
		The distribution of $A_0\cup A_1$ is the same as that of $A$, and thus  $\E(r(A))=\E(r(A_0\cup A_1))$.
		Since  $r(A_0\cup A_1)\geq d+|\hat A_0\cap A_1|$, it follows that \\
		\begin{equation*}
			\E(r(A))\geq d+\E\big(\big|\hat A_0\cap A_1\big|\big)\geq d+(r(E)-d)\frac{i-d}{|E|-d}. \qedhere
		\end{equation*}
	\end{proof}

	\begin{lemma}\label{lem:rc:tbound}
		Let $G=(V,E)$ be a graph, $v\in V$ and let $d\geq 1$. Suppose $r_d(G)\geq r_d(G-v)+d+t$ with $t\geq 0$.
		Then 
		$$ \rc^*_d(G,v)\geq
		d+\frac{t(\deg(v)-d+1)-d(d+1)}{2(\deg(v)+1)}
		.$$
	\end{lemma}
	
	\begin{proof}
		Let $k=\deg(v)$, and let $A_i$ be a subset of $E_v$ of size $i$ chosen uniformly at random for $0\leq i\leq k$.
		Define $\mu_i=\E\big(r_d^v(A_i)).$
		For $0\leq i\leq d$, we have $\mu_i=i$. For $d+1 \leq i \leq k$, an application of Lemma \ref{lem:matroid:randomset} to the matroid $\mathcal{R}_d^v=(E_v,r_d^v)$ gives
 $$\mu_i\geq d+(r_d(G)-r_d(G-v)-d)\frac{i-d}{k-d}\geq d+\frac{t(i-d)}{k-d}.$$ 
		Hence, using \eqref{eq:rc*:equi:form}, we get
		$$(k+1) \rc_d^*(G,v)= \sum_{i=0}^{k} \mu_i\geq  \sum_{i=  0}^d i +\sum_{i=d+1}^{k}\Big(d+\frac{t(i-d)}{k-d}\Big)=(k+1)d-\frac{d(d+1)}{2}+\frac{t(k-d+1)}{2}.$$
		Rearranging the terms completes the proof.
	\end{proof}

    	\begin{lemma}%
		\label{lemma:rc}
		Let $G=(V,E)$ be a graph and $v\in V$. Suppose that $G$ is ${\cal R}_d$-closed, $\deg_G(v)=k\geq d+1$ and
		$G$ contains no $K_{d+2}$ as a subgraph. Then
		$$\rc_d^*(G,v)\geq d+1-\frac{1}{k+1}{\binom{d+2}{2}}. $$
	\end{lemma}
	
	\begin{proof}
		Let $\pi$ be a uniformly random ordering of $V$. If $|N_v^\pi|=i$ for some $0\leq i\leq d$, then $\rc^*_d(G, v, \pi)=i$ by the
		0-extension property (Lemma \ref{ext}). If $|N_v^\pi|\geq d+1$, then $\rc^*_d(G, v, \pi)\geq d+1$ by the
		1-extension property, using that $G$ is ${\cal R}_d$-closed
		and $N_v^\pi$ cannot be complete. Thus, using \eqref{eq:rc*:equi:form}, we obtain
		\begin{equation*}
        \rc_d^*(G,v)\geq \sum_{i=0}^d  \frac{i}{k+1} +
		\sum_{i=d+1}^k \frac{d+1}{k+1}=
		d+1-\frac{1}{k+1}{\binom{d+2}{2}}. \qedhere 
        \end{equation*}
        \end{proof}
	
		For $G=(V,E)$ and $X\subseteq V$, we let $G+K(X)$ denote
	the graph obtained from $G$ by adding an edge $uv$ for
	each non-adjacent vertex pair $u,v\in X$. The following lemma reformulates and slightly strengthens \cite[Lemma 3.2]{vill}. %

	\begin{lemma}\cite{JJV}\label{lem:rc:geq:d}
		Let $ G=( V, E)$ be a graph and $v\in V$ with $\deg_G(v)=k\geq d$. Suppose that $G+K(N_G(v))$ is rigid in $\R^d$ but $G$ is not.
		Then $$\rc_d^*(G,v)\geq d+\frac{1}{2}-\frac{1}{k}{\binom{d+1}{2}}.$$
	\end{lemma}

	\subsection{Proof of Theorem \ref{thm:delta:(n+d)/2-1}}
	\label{subsec:proof+d/2}
	
	In this subsection we prove Theorem \ref{thm:delta:(n+d)/2-1}, which states that if $G$ has at least $29d$ vertices and \begin{equation}\label{eq:n+d-2:proof}
	    \delta(G)\geq \frac{n+d-2}{2},
	\end{equation}
	then $G$ is rigid in $\R^d$.
    The  bound in \eqref{eq:n+d-2:proof} is best possible.

	\begin{proof}
		[Proof of Theorem \ref{thm:delta:(n+d)/2-1}]
		The proof is by induction on $d$. For $d=1$, the statement follows from Lemma \ref{4conn}. Let $d\geq 2$, and suppose, for a contradiction, that
		$G=(V,E)$ is a counter-example with $|V|$ as small as possible,  and subject to this, $|E|$ as large as possible.
		Then $G$ is $\mathcal{R}_d$-closed. 
		\begin{claim}\label{claim:smalln}
			$n\leq d(d+1)-1$.
		\end{claim}
		\begin{proof}
			Suppose for contradiction that $n\geq d(d+1)$. Let $v\in V$ be arbitrary, and consider the graph $G'=G+K(N_G(v))$.
			For every $u\in V-N_G(v)$,  we have $|N_G(u)\cap (N_G(v)\cup \{v\})|\geq d$. Thus, there is a spanning subgraph of $G'$ that can be obtained from $K(N_G(v))$ by a series of 0-extensions, and hence $G'$ is rigid in $\R^d$.
			Using Lemmas \ref{lemma:rc_sum} and \ref{lem:rc:geq:d}, we obtain
			$$
			r_d(G)
			=\sum_{v\in V}\rc_d(G,v)
			\geq \sum_{v\in V}\Big(d+\frac{1}{2}-\frac{d(d+1)}{2\deg(v)}\Big)
			> n\left(d+\frac{1}{2}-\frac{d(d+1)}{2\cdot \frac{n}{2}}\right)\geq nd- \binom{d+1}{2},
			$$
			which is a contradiction.
		\end{proof}

		\begin{claim}\label{claim:DeltaG}
			$G$ has maximum degree $\Delta(G)\leq n-3d-1$.
		\end{claim}
		\begin{proof}
			It follows by induction on $d$ that $G-v$ is rigid in $\R^{d-1}$ for all $v\in V$. Thus
			Theorem \ref{thm:coning}(i) implies that $\Delta(G)\leq n-2$.
			Suppose that $\Delta(G)=n-l$ for some $l\geq 2$, and let $x\in V$ be a vertex with $\deg_G(x)=\Delta(G)$.
			Let $G'$ be the graph obtained from $G$ by adding the set $\{ux:u\in V-N_G(x)-\{x\}\}$ of $l-1$ edges.
			Let $z\in  V-N_G(x)-\{x\}$. By Lemma \ref{lem:dim:addedges}, the edge $xz$ is an $\mathcal{R}_{d+l-2}$-bridge in $G'$.
			Hence, for every $u\in N_G(z)$, the edge $uz$ is an $\mathcal{R}_{d+l-2}$-bridge in $G-x$ by Proposition \ref{prop:easy:bridges}.
			Since $G$ is ${\cal R}_d$-closed, so is $G-x$. Thus we can apply
			Lemma \ref{lem:dofup:s1s2s3}(a) to the graph $G-x$, vertex $v=z$,
			with $d'=d+l-2$ to obtain
			\begin{gather*}
				\rup_{d+l-2}(G-x-z)\geq r_{d+l-2}(G-x)-r_{d+l-2}(G-x-z)+l-2 \\
				\geq\deg_{G-x}(z)+l-2\geq \frac{n+d}{2}+l-3.
			\end{gather*}
			By Lemmas~\ref{lem:dof:coning} and \ref{lem:rup:subgraph} this gives $\rup_{d+l}(G)\geq \rup_{d+l-2}(G-x-z)+2 \geq \frac{n+d}{2}+l-1$. 
			Next we apply Lemma \ref{lem:dof:easybound:v2} with graph $G$, dimension $d+l$ and minimum degree $\frac{n+d}{2}-1$ to obtain 
			$$\rup_{d+l}(G)\left(\frac{n+d}{2}-d-l+1\right)\leq (d+l)(n-d-l+1). $$
			If $l\leq 3d$, then the bounds on $\rup_{d+l}(G)$ deduced above imply
			$\frac{n}{2}(\frac{n}{2}-4d)\leq 4dn$, which gives $n\leq 24d$, contradicting our assumption on $n$. 
			This proves the claim.
		\end{proof}

        Let $s=\lceil 9d/20\rceil$ and $d'=d+s$. Our strategy will be to prove that $G$ is rigid in $\R^{d'}$. Since  rigidity in $\R^{d'}$ implies rigidity in $\R^d$, this will provide a contradiction. Note that by Claim \ref{claim:smalln}, we have $29d\leq n< d(d+1)$, which implies that $29\leq d$.
		
		\begin{claim}\label{claim:dofup19}
			$\rup_{d'}(G)\geq d'+2s-1$.
		\end{claim}
		\begin{proof}
			Let $x\in V$ with $\deg_G(x)=\Delta(G)$. Let $q=\Delta(G)-(\frac{n+d}{2}-1)$ and
			let $Z=V-N_G(x)-\{x\}$. 
			Since $\delta(G)\geq \frac{n}{2}$ and $|Z|\geq 2d$, there exists a vertex $x'\in V$ such that
			$$d\leq |N_G(x')\cap Z|=|N_G(x')-N_G(x)-\{x\}|.$$
			Consider the set $W=N_G(x)\cap N_G(x')$.
			First suppose that
			$|W| \geq d'-1$. Then we can apply Lemma \ref{lem:dofup:s1s2s3}(b) (with $v=x'$, $u=x$, $s_1=s$, $s_2=s-1$, $s_3=d'-1$) to deduce that $\rup_{d'-1}(G-x)\geq d'+2s-2$. From this
			$\rup_{d'}(G)\geq d'+2s-1$ follows by using  Lemmas \ref{lem:dof:coning} and \ref{lem:rup:subgraph}.
			
			Next suppose that $|W| \leq d'-2$. Note that  $|W|\geq \deg_G(x)+\deg_G(x')-n\geq d-2+q$. Hence, using that $0\leq q \leq \frac{n-d}{2}-3d$ by  Claim \ref{claim:DeltaG}, we have
			\begin{gather*}
				\sum_{v\in V} |W- N_G(v)-\{v\}|= \Big|\big\{(v,w)\in V\times W: vw\notin E, v\neq w \big\}\Big|=\sum_{w\in W}|V-N_G(w)-\{w\}|\\
				\geq |W|\cdot (n-\Delta(G)-1)
				\geq (d-2+q)\Big(\frac{n-d}{2}-q\Big)\geq (d-2)\cdot \frac{n-d}{2}.
			\end{gather*} 
			It follows that there exists a vertex $y\in V$ such that
			$$|W- N_G(y)-\{y\}|\geq (d-2)\cdot \frac{n-d}{2n}\geq  (d-2)\cdot \frac{14}{29}\geq\frac{14d}{29}-1> s-2.$$ 
			Since $N_G(x)\cup N_G(x')> n-d$, there exists a vertex $x_0\in \{x,x'\}$ with $|N_G(x_0)\cap N_G(y)|\geq d'$. 
			Furthermore, $$|N_G(x_0)-N_G(y)-\{y\}|\geq |W-N_G(y)-\{y\}|\geq s-1.$$
			Thus,   by applying Lemma \ref{lem:dofup:s1s2s3}(b) with $v=x_0$, $u=y$, we obtain that $\rup_{d'-1}(G-x_0)\geq d'+2s-2$. 
			Then the claim follows by an application of
			Lemmas \ref{lem:dof:coning} and \ref{lem:rup:subgraph}.
		\end{proof}

		For $U\subseteq V$, define $p(U)=r_{d'}(G^{w,U})-r_{d'}(G)$. Then in the matroid $\mathcal{R}_{d'}^w(G^w)=\mathcal{R}_{d'}(G^w) / E(G)$,  the rank of the edge set $\{uw:u\in U\}$ equals $p(U)$. %
		
		\begin{claim}\label{lem:sum:pNGv}
			$\sum_{v\in V}  p(N_G(v))\geq  n(d'+s/3-1) .$
		\end{claim}
		\begin{proof}
			Let $z\in V$ such that $p(N_G(z))$ is minimum. We may assume that $p(N_G(z))< d'+\frac{s}{3}-1$; otherwise, the statement is obvious.
			By Claim  \ref{claim:dofup19}, we have $p(V)\geq d'+2s-1$, and hence, by the matroidal properties of $\mathcal{R}_{d'}^w(G^w)$, there exists a set $X\subseteq V-N_G(z)$ with $|X|\geq \frac{5s}{3}$ such that $$p(N_G(z)\cup X)=p(N_G(z))+|X|.$$ 
			Fix a set $X_0\subseteq X$ with  $|X_0|=s+2$, and let $X_1=X-X_0$.
			Then, $|X_1|\geq \frac{2s}{3}-2$, and
			\begin{equation*}\label{eq:401}
				\sum_{v\in V}|N_G(v)\cap X_1|=\sum_{x\in X_1}\deg_G(x)\geq \frac{n+d-2}{2}|X_1|\geq n\big(\frac{s}{3}-1\big).
			\end{equation*}
			Furthermore, for every $v\in V$, we have $|N_G(v)\cap (N_G(z)\cup X_0)|\geq d'$.
            
            It follows that 
			\begin{equation*}\label{eq:400}
				p(N_G(v))\geq p(N_G(v)\cap  (N_G(z)\cup X_0))+|N_G(v)\cap X_1|\geq d'+|N_G(v)\cap X_1|.
			\end{equation*}
			Thus \begin{equation*}
			    \displaystyle\sum_{v\in V}  p(N_G(v))\geq nd'+\sum_{v\in V}|N_G(v)\cap X_1|\geq  n\big(d'+\frac{s}{3}-1\big).\qedhere
			\end{equation*}
		\end{proof}
		
		We can now complete the proof of the theorem as follows.
		For $v\in V$, define  $t(v)= r_{d'}(G)-r_{d'}(G-v)-d'$. Let $H=G^{w,N_G(v)}$. Then the graph $H-v$ is isomorphic to $G$.
		By the submodularity of $r_{d'}$, we have  $$r_{d'}(G)+r_{d'}(H-v)\geq r_{d'}(H)+r_{d'}(G-v),$$ which implies that
		$t(v)\geq p(N_G(v))-d'$ for every $v\in V$.
        
		Thus, by Claim \ref{lem:sum:pNGv}, 
		\begin{gather*}
			\sum_{v\in V} \frac{t(v)(\deg(v)-d')}{2(\deg(v)+1)}\geq \frac{\delta(G)-d'}{\delta(G)+1}\sum_{v\in V}\frac{p(N_G(v))-d'}{2}
			\geq \frac{n-2d}{n+d}\bigg( \frac{ns}{6}-\frac{n}{2}\bigg)\geq\\
			\geq \left(1-\frac{3d}{n}\right)\bigg( \frac{ns}{6}-\frac{n}{2}\bigg)
			\geq \frac{ns}{6}-\Big(\frac{n}{2}+\frac{3d^2}{12}\Big)\geq \frac{ns}{6}-\binom{d'+1}{2},
		\end{gather*}
		where in the last inequality we use Claim \ref{claim:smalln}.
		
		Applying Lemma \ref{lem:rc:tbound}, we obtain
		\begin{gather*}
			r_{d'}(G)=\sum_{v\in V} \rc_{d'}(G)\geq \sum_{v\in V}\bigg( d'+ \frac{t(v)(\deg(v)-d')-d'(d'+1)}{2(\deg(v)+1)}\bigg)\geq \\
			\geq nd'+\frac{ns}{6}-\binom{d'+1}{2}- \frac{nd'(d'+1)}{n+d}\geq nd'-\binom{d'+1}{2},
		\end{gather*}
		where the last inequality follows from $\frac{s}{6}-\frac{d'(d'+1)}{n+d}\geq \frac{9d}{6\cdot 20}-\frac{(3d/2)^2}{30d}=0$. Hence, $G$ is rigid in $\R^{d'}$, a contradiction.
	\end{proof}
	
	\section{Degree sum conditions in higher dimensions}
	\label{sec:degreesum}
	
	In this section we obtain sufficient conditions for
	$d$-rigidity in terms of lower bounds on the degree sum
	of non-adjacent vertex pairs. %
	
	\subsection{A general bound for graphs of arbitrary order}
	We start with the proof  of Theorem \ref{thm:intro:g:5d}, which states that if a graph $G$ satisfies
	\begin{equation}
		\label{alln:degsum:3d}
		\deg_G(u)+\deg_G(v)\geq n + 3d-3
	\end{equation}
	for all pairs of non-adjacent vertices, then $G$
	is rigid in $\R^d$.

	\begin{proof}[Proof of Theorem \ref{thm:intro:g:5d}]
		Let $G=(V,E)$ be a graph that satisfies \eqref{alln:degsum:3d} for all non-adjacent pairs $u,v\in V$.
		We prove that $G$ is rigid in $\R^d$. By replacing $G$ with its closure, we may assume without loss of generality that $G$ is $\mathcal{R}_d$-closed.
		
		Let $Y=\{ v\in V: \deg_G(v)< \frac{n+3d-3}{2}\}$ be the set of {\it small degree vertices}. Then $G[Y]$ is complete by \eqref{alln:degsum:3d} (and may be empty).
		Let $K\subseteq V$ such that $G[K]$ is a maximal complete subgraph of $G$ with $Y\subseteq K$.
		Set $L=V-K$. Then for each $v\in L$, we have  $\deg_G(v)\geq \frac{n+3d-3}{2}$ and $e_G(v,K)\leq \min(d-1, |K|)$.
		Hence, $$\delta(G[L])\geq \frac{n+3d-3}{2}-\min(d-1, |K|)\geq \frac{n+3d-3}{2}-\frac{|K|+d-1}{2}=\frac{|L|}{2}+d-1.$$
		By Theorem \ref{thm:delta:n/2+d}, this implies that $G[L]$ is rigid.
		Using that $G$ is $\mathcal{R}_d$-closed, it follows that $G[L]$ is a clique.
		If $G$ is complete, then we are done. Otherwise, there exist $u\in K$ and $v\in L$ such that $uv\notin E$.
		From \eqref{alln:degsum:3d} it follows that $u$ and $v$ have at least
        $3k-1\geq 2d$ common neighbours in $G$, and hence
              $e_G(u,L)\geq d$ or $e_G(v,K)\geq d$.
		Therefore either $G[L\cup \{u\}]$ or $G[K\cup \{v\}]$ is a non-complete rigid subgraph of $G$, contradicting the assumption that $G$ is $\mathcal{R}_d$-closed.
	\end{proof}
	
	\subsection{Tight bound for sufficiently large graphs}
    
	While Theorem \ref{thm:intro:g:5d} establishes the bound $g(n,d)\leq n+3d-3$ for all $n$ and $d$, a gap remains between this and the lower bound
	\begin{equation}\label{eq:lowerbound:gnd}
		n + d - 2\leq g(n,d).
	\end{equation}

    We believe that this lower bound is tight for all
    $n>2d+2$ (see
	Conjecture \ref{gnd:conj2}).
	The goal of this subsection is to prove Theorem \ref{thm:intro:g:d-1}, which states that $g(n,d)= n + d - 2$ for all $d$ and all sufficiently large $n$. This result generalizes \cite[Theorem 1.1]{KLM} with a slightly improved bound on $n$. In fact, our proof follows a similar approach, based on the following observation as well as the rank contribution of vertices in $\mathcal{R}_d$.

	\begin{lemma}\label{lem:simpl:vert}
		Let $G=(V,E)$ be a graph. Suppose that $\deg(u)+\deg(v)\geq n+d-2$ for all pairs of non-adjacent vertices in $G$. If there is a vertex $w\in V$ 
        such that $\{u,v\}$ is ${\cal R}_d$-linked
        for all $u,v\in N_G(w)$,
then $G$ is rigid in $\R^d$.
	\end{lemma}
	\begin{proof}
		Let $G'$ denote the $\mathcal{R}_d$-closure of $G$, and let $S$ denote a maximal clique of $G'$ that contains $N_G(w)\cup \{w\}$. Suppose, for a contradiction, that there is a vertex $v\in V-V(S)$. Note that $vw\notin E$, and hence $|V(S)\cap N_{G'}(v)|\geq |N_G(w)\cap N_G(v)|\geq d$. Thus the 0-extension property of $\mathcal{R}_d$ (Lemma \ref{ext}) implies that $G'[V(S)\cup \{v\}]$ is rigid, which contradicts the maximality of $S$. Therefore, $V(S)=V$, and $G'$ is a complete graph. It follows that $G$ is rigid.
	\end{proof}

	\begin{lemma}
		\label{minlarge}
		Let $G=(V,E)$ be a graph on $n\geq 2d+1$ vertices with $\delta(G)=d$.
		Suppose that $\deg(u)+\deg(v)\geq n+d-2$ for all pairs of non-adjacent
		vertices in $G$. Then $G$ is rigid in $\R^d$.
	\end{lemma}
	
	\begin{proof}
		Let $v\in V$ be a vertex of degree $d$ and
		let $T=V-N(v)-\{v\}$. Then each vertex $w\in T$ has degree $n-2$.
		Thus $G[T]$ is a complete graph on at least $d$ vertices and $G$
		has a spanning subgraph $G'$ that can be obtained from $G[T]$ by a
		sequence of $0$-extensions. Hence $G$ is rigid.
	\end{proof}

		\begin{proof}[Proof of Theorem \ref{thm:intro:g:d-1}]
		Let $G=(V,E)$ be a graph on at least $n\geq d(d+2)$ vertices that satisfies
		\begin{equation}
			\label{alln:degsum}
			\deg(u)+\deg(v)\geq n + d - 2
		\end{equation}
		for all pairs of non-adjacent vertices. We shall prove that $G$ is $d$-rigid. Suppose it is not.
		
		It follows from Lemma \ref{lem:simpl:vert} that for every $v\in V$, the graph $G+K(N_G(v))$ is rigid in $\R^d$.
		Let $k=\delta(G)$ denote the smallest degree of $G$.  We have $k\geq d+1$ by Lemma \ref{minlarge}.
		Hence, Lemmas \ref{lemma:rc_sum} and \ref{lem:rc:geq:d} imply that
		\begin{equation}\label{fq1:Ndproof}
			r_d(G)
			=\sum_{v\in V}\rc_d(G,v)
			\geq \sum_{v\in V}\Big(d+\frac{1}{2}-\frac{d(d+1)}{2\deg(v)}\Big)
			=dn+\frac{n}{2}-\frac{d(d+1)}{2}\sum_{v\in V}\frac{1}{\deg(v)}.
		\end{equation}
		Let $x_0\in V$ be a vertex with $\deg(x_0)=k$, and define
		$X=N_G(x_0)\cup \{x_0\}$ and $Y=V-X$. Set $l=n+d-k-2$. Then $|Y|=l-d+1$, and, for all $y\in Y$, since $yx_0\notin E$, we have $\deg(y)\geq l$.  Using that $\deg(x)\geq k$ for all $x\in X$, we obtain
		\begin{equation}\label{fq2:Ndproof}
			\sum_{v\in V}\frac{1}{\deg(v)}\leq \frac{|X|}{k}+ \frac{|Y|}{l}=\frac{k+1}{k}+\frac{l-d+1}{l}
			\leq 2+\frac{1}{k}\leq\frac{2d+3}{d+1},
		\end{equation}
		where the last inequality follows from $k\geq d+1$.
		Now \eqref{fq1:Ndproof} and \eqref{fq2:Ndproof} together imply that $$r_d(G)\geq dn+\frac{n}{2}-\frac{d(2d+3)}{2}= dn+\frac{n-d(d+2)}{2}-\frac{d(d+1)}{2}\geq dn-\frac{d(d+1)}{2}.
		$$
		This contradicts the assumption that $G$ is not rigid.
	\end{proof}

	\section{Sharp results in $\R^2$ and $\R^3$}
	\label{sec:lowdim}
	
	In this section, we study the cases $d=2,3$ in detail.
	Our main result is Theorem \ref{rigid graph in R^3} below, which characterizes 
	those graphs $G$ that satisfy $\eta(G)\geq n+1$ but are not 3-rigid.  This theorem enables us to determine $f(n,d)$ and $g(n,d)$ for all $n$ and $d=2,3$. 
	As a consequence, it confirms Conjectures \ref{KLMconj} and \ref{gnd:conj2} in these dimensions.
	We first state the theorem and collect its corollaries. Its proof is presented in a separate subsection. 

	Let $W_5$ be a wheel on five vertices and let $B_6$ be the graph
	obtained from $K_6$ by deleting the edges of two disjoint
	paths on three vertices. 	Let $C^1_7$ (resp.\ $C^2_7$) be the graph obtained from $K_7$ by removing the edges of two disjoint cycles, one of length three and one of length four
	(resp.\ by removing the edges of a cycle on seven vertices). See Fig. \ref{fig}. These graphs all have fewer than $3n-6$ edges, and hence they are not rigid in $\R^3$.
	
	\begin{figure}
		\centering
		\begin{pgfpicture}{-22.62mm}{110.60mm}{117.62mm}{139.06mm}
			\pgfsetxvec{\pgfpoint{1.00mm}{0mm}}
			\pgfsetyvec{\pgfpoint{0mm}{1.00mm}}
			\color[rgb]{0,0,0}\pgfsetlinewidth{0.30mm}\pgfsetdash{}{0mm}
			\pgfcircle[fill]{\pgfxy(-9.95,135.09)}{0.50mm}
			\pgfsetlinewidth{0.15mm}\pgfcircle[stroke]{\pgfxy(-9.95,135.09)}{0.50mm}
			\pgfcircle[fill]{\pgfxy(-20.12,124.95)}{0.50mm}
			\pgfcircle[stroke]{\pgfxy(-20.12,124.95)}{0.50mm}
			\pgfcircle[fill]{\pgfxy(-10.00,124.84)}{0.50mm}
			\pgfcircle[stroke]{\pgfxy(-10.00,124.84)}{0.50mm}
			\pgfcircle[fill]{\pgfxy(25.01,134.96)}{0.50mm}
			\pgfcircle[stroke]{\pgfxy(25.01,134.96)}{0.50mm}
			\pgfcircle[fill]{\pgfxy(-15.12,129.92)}{0.50mm}
			\pgfcircle[stroke]{\pgfxy(-15.12,129.92)}{0.50mm}
			\pgfsetlinewidth{0.30mm}\pgfmoveto{\pgfxy(-20.05,134.78)}\pgflineto{\pgfxy(-20.08,124.96)}\pgfstroke
			\pgfmoveto{\pgfxy(-20.05,134.93)}\pgflineto{\pgfxy(-10.00,134.92)}\pgfstroke
			\pgfmoveto{\pgfxy(-9.97,135.03)}\pgflineto{\pgfxy(-9.97,124.95)}\pgfstroke
			\pgfmoveto{\pgfxy(-20.05,124.95)}\pgflineto{\pgfxy(-20.05,124.95)}\pgfstroke
			\pgfmoveto{\pgfxy(-20.05,124.91)}\pgflineto{\pgfxy(-9.86,124.87)}\pgfstroke
			\pgfmoveto{\pgfxy(-20.05,134.85)}\pgflineto{\pgfxy(-10.07,124.95)}\pgfstroke
			\pgfmoveto{\pgfxy(-10.00,135.03)}\pgflineto{\pgfxy(-20.05,124.91)}\pgfstroke
			\pgfputat{\pgfxy(-17.24,113.47)}{\pgfbox[bottom,left]{\fontsize{11.38}{13.66}\selectfont $W_5$}}
			\pgfmoveto{\pgfxy(-20.05,134.85)}\pgflineto{\pgfxy(-10.07,124.95)}\pgfstroke
			\pgfcircle[fill]{\pgfxy(15.02,124.94)}{0.50mm}
			\pgfsetlinewidth{0.15mm}\pgfcircle[stroke]{\pgfxy(15.02,124.94)}{0.50mm}
			\pgfcircle[fill]{\pgfxy(4.99,124.97)}{0.50mm}
			\pgfcircle[stroke]{\pgfxy(4.99,124.97)}{0.50mm}
			\pgfcircle[fill]{\pgfxy(-20.02,134.96)}{0.50mm}
			\pgfcircle[stroke]{\pgfxy(-20.02,134.96)}{0.50mm}
			\pgfcircle[fill]{\pgfxy(15.03,134.96)}{0.50mm}
			\pgfcircle[stroke]{\pgfxy(15.03,134.96)}{0.50mm}
			\pgfcircle[fill]{\pgfxy(44.99,135.02)}{0.50mm}
			\pgfcircle[stroke]{\pgfxy(44.99,135.02)}{0.50mm}
			\pgfcircle[fill]{\pgfxy(4.96,134.96)}{0.50mm}
			\pgfcircle[stroke]{\pgfxy(4.96,134.96)}{0.50mm}
			\pgfcircle[fill]{\pgfxy(25.03,125.05)}{0.50mm}
			\pgfcircle[stroke]{\pgfxy(25.03,125.05)}{0.50mm}
			\pgfsetlinewidth{0.30mm}\pgfmoveto{\pgfxy(15.06,134.84)}\pgflineto{\pgfxy(15.06,124.94)}\pgfstroke
			\pgfmoveto{\pgfxy(4.96,134.98)}\pgflineto{\pgfxy(4.96,125.08)}\pgfstroke
			\pgfmoveto{\pgfxy(24.99,134.76)}\pgflineto{\pgfxy(24.99,124.86)}\pgfstroke
			\pgfmoveto{\pgfxy(4.98,134.90)}\pgflineto{\pgfxy(15.06,124.93)}\pgfstroke
			\pgfmoveto{\pgfxy(15.01,134.97)}\pgflineto{\pgfxy(25.08,125.00)}\pgfstroke
			\pgfmoveto{\pgfxy(5.05,134.94)}\pgflineto{\pgfxy(25.08,124.97)}\pgfstroke
			\pgfmoveto{\pgfxy(15.06,134.98)}\pgflineto{\pgfxy(4.94,124.95)}\pgfstroke
			\pgfmoveto{\pgfxy(25.00,135.00)}\pgflineto{\pgfxy(5.03,125.00)}\pgfstroke
			\pgfmoveto{\pgfxy(25.08,134.85)}\pgflineto{\pgfxy(14.97,124.82)}\pgfstroke
			\pgfputat{\pgfxy(13.43,113.49)}{\pgfbox[bottom,left]{\fontsize{11.38}{13.66}\selectfont $B_6$}}
			\pgfcircle[fill]{\pgfxy(50.01,124.97)}{0.50mm}
			\pgfsetlinewidth{0.15mm}\pgfcircle[stroke]{\pgfxy(50.01,124.97)}{0.50mm}
			\pgfcircle[fill]{\pgfxy(39.95,124.94)}{0.50mm}
			\pgfcircle[stroke]{\pgfxy(39.95,124.94)}{0.50mm}
			\pgfcircle[fill]{\pgfxy(65.05,135.00)}{0.50mm}
			\pgfcircle[stroke]{\pgfxy(65.05,135.00)}{0.50mm}
			\pgfcircle[fill]{\pgfxy(59.92,125.02)}{0.50mm}
			\pgfcircle[stroke]{\pgfxy(59.92,125.02)}{0.50mm}
			\pgfcircle[fill]{\pgfxy(70.00,125.13)}{0.50mm}
			\pgfcircle[stroke]{\pgfxy(70.00,125.13)}{0.50mm}
			\pgfcircle[fill]{\pgfxy(55.02,135.01)}{0.50mm}
			\pgfcircle[stroke]{\pgfxy(55.02,135.01)}{0.50mm}
			\pgfsetlinewidth{0.30mm}\pgfmoveto{\pgfxy(4.97,135.00)}\pgflineto{\pgfxy(15.04,135.00)}\pgfstroke
			\pgfmoveto{\pgfxy(4.94,124.88)}\pgflineto{\pgfxy(15.01,124.88)}\pgfstroke
			\pgfmoveto{\pgfxy(59.82,125.07)}\pgflineto{\pgfxy(70.00,125.09)}\pgfstroke
			\pgfmoveto{\pgfxy(45.07,135.06)}\pgflineto{\pgfxy(39.82,124.92)}\pgfstroke
			\pgfmoveto{\pgfxy(55.19,135.02)}\pgflineto{\pgfxy(49.93,124.86)}\pgfstroke
			\pgfmoveto{\pgfxy(65.12,135.05)}\pgflineto{\pgfxy(59.86,124.89)}\pgfstroke
			\pgfmoveto{\pgfxy(44.98,135.04)}\pgflineto{\pgfxy(50.04,125.00)}\pgfstroke
			\pgfmoveto{\pgfxy(39.93,124.95)}\pgflineto{\pgfxy(50.03,124.95)}\pgfstroke
			\pgfmoveto{\pgfxy(55.01,135.04)}\pgflineto{\pgfxy(60.08,125.00)}\pgfstroke
			\pgfmoveto{\pgfxy(65.17,135.16)}\pgflineto{\pgfxy(70.23,125.12)}\pgfstroke
			\pgfmoveto{\pgfxy(45.00,135.00)}\pgflineto{\pgfxy(59.97,125.02)}\pgfstroke
			\pgfmoveto{\pgfxy(45.00,135.05)}\pgflineto{\pgfxy(70.05,125.05)}\pgfstroke
			\pgfmoveto{\pgfxy(55.02,135.00)}\pgflineto{\pgfxy(39.89,125.08)}\pgfstroke
			\pgfmoveto{\pgfxy(55.02,134.97)}\pgflineto{\pgfxy(70.16,125.26)}\pgfstroke
			\pgfmoveto{\pgfxy(65.05,135.00)}\pgflineto{\pgfxy(50.02,124.98)}\pgfstroke
			\pgfmoveto{\pgfxy(65.03,135.05)}\pgflineto{\pgfxy(39.92,124.95)}\pgfstroke
			\pgfputat{\pgfxy(53.32,113.61)}{\pgfbox[bottom,left]{\fontsize{11.38}{13.66}\selectfont $C_7^1$}}
			\pgfcircle[fill]{\pgfxy(90.08,135.02)}{0.50mm}
			\pgfsetlinewidth{0.15mm}\pgfcircle[stroke]{\pgfxy(90.08,135.02)}{0.50mm}
			\pgfcircle[fill]{\pgfxy(95.10,124.97)}{0.50mm}
			\pgfcircle[stroke]{\pgfxy(95.10,124.97)}{0.50mm}
			\pgfcircle[fill]{\pgfxy(85.03,124.94)}{0.50mm}
			\pgfcircle[stroke]{\pgfxy(85.03,124.94)}{0.50mm}
			\pgfcircle[fill]{\pgfxy(110.14,135.00)}{0.50mm}
			\pgfcircle[stroke]{\pgfxy(110.14,135.00)}{0.50mm}
			\pgfcircle[fill]{\pgfxy(105.01,125.02)}{0.50mm}
			\pgfcircle[stroke]{\pgfxy(105.01,125.02)}{0.50mm}
			\pgfcircle[fill]{\pgfxy(115.12,124.99)}{0.50mm}
			\pgfcircle[stroke]{\pgfxy(115.12,124.99)}{0.50mm}
			\pgfcircle[fill]{\pgfxy(100.11,135.01)}{0.50mm}
			\pgfcircle[stroke]{\pgfxy(100.11,135.01)}{0.50mm}
			\pgfsetlinewidth{0.30mm}\pgfmoveto{\pgfxy(104.95,124.96)}\pgflineto{\pgfxy(115.13,124.98)}\pgfstroke
			\pgfmoveto{\pgfxy(90.15,135.06)}\pgflineto{\pgfxy(84.90,124.92)}\pgfstroke
			\pgfmoveto{\pgfxy(110.21,135.05)}\pgflineto{\pgfxy(104.95,124.89)}\pgfstroke
			\pgfmoveto{\pgfxy(90.07,135.04)}\pgflineto{\pgfxy(95.13,125.00)}\pgfstroke
			\pgfmoveto{\pgfxy(85.02,124.95)}\pgflineto{\pgfxy(95.12,124.95)}\pgfstroke
			\pgfmoveto{\pgfxy(110.26,135.16)}\pgflineto{\pgfxy(115.32,125.12)}\pgfstroke
			\pgfmoveto{\pgfxy(100.11,135.00)}\pgflineto{\pgfxy(84.98,125.08)}\pgfstroke
			\pgfmoveto{\pgfxy(100.11,134.97)}\pgflineto{\pgfxy(115.24,125.26)}\pgfstroke
			\pgfputat{\pgfxy(98.40,113.61)}{\pgfbox[bottom,left]{\fontsize{11.38}{13.66}\selectfont $C_7^2$}}
			\pgfmoveto{\pgfxy(89.96,135.04)}\pgflineto{\pgfxy(110.06,135.04)}\pgfstroke
			\pgfmoveto{\pgfxy(90.00,135.06)}\pgfcurveto{\pgfxy(93.15,136.36)}{\pgfxy(96.52,137.04)}{\pgfxy(99.93,137.06)}\pgfcurveto{\pgfxy(103.40,137.08)}{\pgfxy(106.83,136.42)}{\pgfxy(110.04,135.11)}\pgfstroke
			\pgfmoveto{\pgfxy(95.11,124.96)}\pgflineto{\pgfxy(105.28,124.98)}\pgfstroke
			\pgfmoveto{\pgfxy(84.94,124.91)}\pgfcurveto{\pgfxy(88.15,123.60)}{\pgfxy(91.58,122.93)}{\pgfxy(95.05,122.93)}\pgfcurveto{\pgfxy(98.49,122.93)}{\pgfxy(101.89,123.58)}{\pgfxy(105.08,124.85)}\pgfstroke
			\pgfmoveto{\pgfxy(94.80,124.80)}\pgfcurveto{\pgfxy(98.07,123.55)}{\pgfxy(101.54,122.92)}{\pgfxy(105.04,122.93)}\pgfcurveto{\pgfxy(108.43,122.94)}{\pgfxy(111.78,123.56)}{\pgfxy(114.94,124.75)}\pgfstroke
		\end{pgfpicture}
		\mycaption{}
		\label{fig}
	\end{figure}

		\begin{theorem}
		\label{rigid graph in R^3}
		Let $G=(V,E)$ be a graph on $n$ vertices. Suppose that $G$ is not isomorphic to
		$W_5$, $B_6$, $C_7^1$, or $C_7^2$. If
		\begin{equation*}
			\label{3dimthm}
			\deg(u)+\deg(v)\geq n+1
		\end{equation*}
		for all pairs of non-adjacent vertices in $G$, then
		 $G$ is rigid in $\R^3$.
	\end{theorem}

		Using Theorem \ref{rigid graph in R^3}, we can deduce Conjectures \ref{KLMconj} and \ref{gnd:conj2} in the case $d=3$.
		
		\begin{cor}
			\label{conjd3}
			Let $n\geq 4$. If $n\in \{5,6,7\}$, then $g(n,3)=n+2$. Otherwise, $g(n,3)=n+1$. 
			Moreover,
			 \begin{equation}\label{eq:fn3}
			 	f(n,3)=  \max \Big\{ \Big\lceil\frac{n+1}{2}\Big\rceil, \Big\lceil 6-\frac{12}{n}\Big\rceil \Big\}.
			 \end{equation}
		\end{cor}
		
		\begin{proof}
			The 
upper bounds for
$g(n,3)$ and
$f(n,3)$
directly follow from Theorem \ref{rigid graph in R^3}.
			For $n\geq 8$, 
            Lemma \ref{4conn} implies that these bounds are tight. For $n\leq 7$, this is demonstrated by the graphs $K_4-e,W_5, B_6, C_7^1$.
		\end{proof}
		
		 Corollary \ref{conjd3} shows that for any $n$, the bound $f(n,3)\leq \lceil \frac{g(n,3)}{2}\rceil$ is tight.
		 
		In the rest of this subsection, we consider the case $d=2$.
		It turns out that Theorem \ref{rigid graph in R^3} can  also be applied to answer the corresponding 2-dimensional questions.

		\begin{theorem}
			\label{rigid graphs in R2}
			Let $G=(V,E)$ be a graph on $n$ vertices such that $G\neq C_4$.
			If
			\begin{equation}
				\label{2dim}
				\deg(u)+\deg(v)\geq n
			\end{equation}
			for all pairs of non-adjacent vertices in $G$, then
			$G$ is rigid in $\R^2$.
		\end{theorem}
		
		\begin{proof}
			Suppose that $G$ satisfies (\ref{2dim}) but $G$ is not rigid in $\R^2$.
			Then $G^w$, the cone of $G$, satisfies (\ref{3dim}), and it is not
			rigid in $\R^3$ by Whiteley's coning theorem.
			So $G^w$ is isomorphic to one of the graphs $\{W_5, B_6, C_7^1, C_7^2\}$ from
			Theorem \ref{rigid graph in R^3}.
			The only cone graph in this list 
			is $W_5$, which is the cone of $C_4$. However, we assumed that $G\not=C_4$.
		\end{proof}

		Theorem \ref{rigid graphs in R2} implies an affirmative answer to Conjectures \ref{KLMconj} and \ref{gnd:conj2} for
		$d=2$.
		The proof is similar to that of Corollary \ref{conjd3}.
		
		\begin{cor}
			\label{conjd2}
			Let $n\geq 3$. If $n=4$, then $g(n,2)=n+1$. Otherwise, $g(n,2)=n$. Moreover, $$f(n,2)=  \max \Big\{ \Big\lceil\frac{n}{2}\Big\rceil, \Big\lceil 4-\frac{6}{n}\Big\rceil \Big\}.$$
		\end{cor}
		
		\subsection{The proof of Theorem \ref{rigid graph in R^3}}

		The proof of Theorem \ref{rigid graph in R^3} will follow an inductive argument.
		We start with two technical lemmas that are concerned with graphs on at most ten vertices.
		
		\begin{lemma}\label{claim2}
			Let $G=(V,E)$ be a graph on $n$ vertices with $\delta(G)\geq 4$
			and $6\leq n\leq 7$. Then either $G$ is
			isomorphic to $C_7^1$ or $C_7^2$, or $G$ is rigid in $\R^3$.
		\end{lemma}
		
		\begin{proof}
			Suppose that $G$ is not isomorphic to $C_7^1$
			or $C_7^2$. Then $|E|\geq 3n-6$. 
			Since the complement of $G$ has maximum degree at most two, $G$ has a spanning
			subgraph $H$ with $\delta(H)\geq 4$ and $|E(H)|=3n-6$.
			We claim that $H$ is ${\cal R}_3$-independent, from which the
			rigidity of $G$ follows.
			Suppose not. Then $H$ contains an ${\cal R}_3$-circuit $D$ as a subgraph.
			We have $5\leq |V(D)|\leq n\leq 7$.
			It is well-known that every ${\cal R}_3$-circuit on at most seven vertices
			is rigid in $\R^3$. Thus $|E(D)|=3|V(D)|-5$. Then $\delta(H) \geq 4$
			gives $|E(H)|>3n-6$, a contradiction. 
		\end{proof}
		
		\begin{lemma}
			\label{claim3}
			Let $G=(V,E)$ be a graph on $n$ vertices with $\delta(G)=4$ and
			$8\leq n\leq 10$. Suppose that
			\begin{equation}
				\label{eq}
				\deg(u)+\deg(v)\geq n+1
			\end{equation}
			for all pairs of non-adjacent vertices $u,v$ in $G$. Then
			$G$ contains a subgraph on at least four vertices which is
			rigid in $\R^3$.
		\end{lemma}

		\begin{proof}
			Let $w$ be a vertex with $\deg(w)=\delta(G)=4$, $K=N_G(w)$, and
			$T=V-K-\{w\}$. Then $3\leq |T|\leq 5$ and 
			$\deg(v)\geq |T|+2$ for each vertex $v\in T$ by (\ref{eq}).
			Furthermore, we can observe that
			for each pair of non-adjacent vertices $u,v$ in $T$, $u$ is adjacent to all vertices in $G$ except for $w$ and $v$.
						
			First suppose that $|T|=3$. Since $\deg(t)\geq 5$ for each $t\in T$,
			there is a vertex $k\in K$
			which is adjacent to all vertices in $T$. So if $G[T]=K_3$, then $G$
			contains a $K_4$. Otherwise (\ref{eq}) and the observation above
			implies that
			$G[T]$ is a $K_{1,2}$ with center vertex $t$, 
			$e(t,K)\geq 3$, 
			and $e(t',K)=4$ for each vertex $t'\in T-\{t\}$.
			Moreover, $G[K]$ has at least two edges. These facts imply that
			$G$ can be obtained from a $K_4$ on vertex set $K$
			by four $1$-extensions and edge additions. Thus $G$ is rigid in $\R^3$.
			
			Next suppose $|T|=4$. If $G[T]=K_4$, we are done.
			Otherwise
			$G[T]$ is a graph obtained from $K_4$ by removing at most two disjoint edges. If $G[T]$ contains exactly four edges, then $e(k,T)=4$ for each $v\in K$.
			Let $k,k'\in K$. Then (a spanning subgraph of)
			$G[T\cup \{k,k'\}]$ can be obtained from a $K_4$ on vertex set $T$ by two 
			$1$-extensions, so it is rigid. If $G[T]$ contains exactly five edges, then $G[T]$ contains a $K_3$, whose vertices have a common neighbor in $K$.
			So $G$ contains a $K_4$.
			
			Finally, suppose that
			$|T|=5$. If $G[T]$ contains a $K_4$, we are done. Otherwise the observation
			above implies that
			$G[T]$ is a graph obtained from $K_5$ by removing two disjoint edges. Furthermore, all but one of the vertices in $T$ are connected to $K$ by four edges and each vertex of $T$ has degree at least seven in $G$.
			Therefore (a spanning subgraph of) $G$ can be obtained from $K_4$ on vertex set $K$ by six $1$-extensions. So $G$ is rigid in $\R^3$.
		\end{proof}
		
		In the special case $d=3$,  Lemma \ref{minlarge} easily extends to $n=5,6$:
		
		\begin{lemma}
			\label{minlarge3}
			Let $G=(V,E)$ be a graph on $n\geq 5$ vertices with $\delta(G)=3$.
			Suppose that $\deg(u)+\deg(v)\geq n+1$ for all pairs of non-adjacent
			vertices in $G$. Then either $G$ is rigid in $\R^3$, or $G$ is
			isomorphic to $W_5$ or $B_6$.
		\end{lemma}
		
		We are now ready to prove Theorem \ref{rigid graph in R^3}.
		
		\begin{proof}[Proof of Theorem \ref{rigid graph in R^3}]
			Suppose for a contradiction  that there exists a graph $G$ that is not rigid in $\R^d$, is not isomorphic to $W_5$, $B_6$, $C_7^1$, or $C_7^2$, and satisfies 
			\begin{equation}
				\label{3dim}
				\deg(u)+\deg(v)\geq n+1
			\end{equation}
			for all non-adjacent pairs $u,v\in V$.
			Let us choose  $G$ such that $|V|$ is as small as possible and,
			subject to this, $|E|$ is as large as possible.
			Then $G$ is ${\cal R}_3$-closed.
			By Lemmas \ref{4conn} and \ref{minlarge3}, we must have $\delta(G)\geq 4$.
			Let $G'$ be a maximum size rigid (or equivalently, complete) subgraph of $G$,
			let $C=V(G')$, $D=V-C$, $H=G[D]$ and $q=|C|$.
			
			\begin{claim}
				\label{mrg}
				$q\geq 4$.
			\end{claim}

			\begin{proof}
				For a contradiction suppose that $G$ contains no $K_4$.
				Let
				$k=\delta(G)$ and
				let
				$u\in V$ with $\deg(u)=k$.
				Lemma \ref{lemma:rc} implies that for each $v\in V$
				we have $\rc_3(G,v)\geq 4 - \frac{10}{\deg(v)+1}.$ Let $T=V-N_G(u)-\{u\}$. Then $|T|=n-k-1$ and, by  (\ref{3dim}), for all $t\in T$ we have
				$\deg(t)\geq n-k+1$.
				Hence,
				\begin{equation}
					\begin{aligned}
						\sum_{v\in V} \rc_3(G,v) %
						&=\sum_{v\in V-T} \big(4-\frac{10}{\deg(v)+1}\big)+\sum_{v\in T}
						\big(4-\frac{10}{\deg(v)+1}\big)\\
						&\geq 4n-(k+1)\frac{10}{k+1} - (n-k-1)\frac{10}{n-k+2}=
						4n-20+
						\frac{30}{n-k+2}.\notag
					\end{aligned}
				\end{equation}
				A simple computation shows that either the right hand side is at least
				$3n-6$, or  $k=4$ and $6\leq n \leq 10$.
				In the former case we can use
				Lemma \ref{lemma:rc} to obtain
				$$r_3(G)\geq  \sum_{v\in V} \rc_3(G,v) \geq 3n-6,$$
				which contradicts the assumption that $G$ is not rigid.
				In the latter case
				Lemmas \ref{claim2} and \ref{claim3} imply that $G$ contains a $K_4$.
			\end{proof}

			\begin{claim}
				$q\leq \frac{n-1}{2}$ and $H$ is not rigid.
			\end{claim}

			\begin{proof}
				By the maximality of $C$
				we have $e(v,C)\leq 2$ for all $v\in D$ and (\ref{3dim}) implies that $e(u,D)\geq 1$ for all $u\in C$.
				
				First suppose that $H$ is complete. Clearly, $|D|\geq 3$ as the minimum degree is at least four. If $e(v,C)=2$ for all $v\in D$, then we can use
				Lemma \ref{cycle} to deduce that $G$ is rigid, a contradiction.
				Thus there is a vertex $v'\in D$ with $e(v',C)\leq 1$. Furthermore, 
				Lemma \ref{union} and the maximality of $H$ imply that there exists
				a subset $A\subseteq C$ with $|A|\leq 2$ such that $e(u,D)= 1$
				for all $u\in C-A$. Let $u'\in C-A$ be a vertex which is not adjacent to $v'$.
				Since $q\geq 4$, such a vertex exists.
				Then we have
				$$\deg(u')+\deg(v')\leq q+(n-q)=n,$$
				a contradiction.
				
				Therefore, $H$ is not complete (equivalently, $H$ is not rigid). Thus $D$ contains a pair of non-adjacent vertices $x$ and $y$. Then (\ref{3dim}) gives
				$$n+1\leq \deg(x)+\deg(y)\leq 2(n-q-2+2)=2n-2q,$$
				which implies $q\leq \frac{n-1}{2}$, as required.
			\end{proof}
			
			\begin{claim}
				\label{claim1}
				$H$ is not isomorphic to any of $W_5$, $B_6$, $C_7^1$ or $G_7^2$. 
			\end{claim}
			
			\begin{proof}
				In the proof we shall use the fact that $C_7^1$, $C_7^2$, and every vertex-deleted
				subgraph of $B_6$ has a cycle that contains all vertices.
				We shall use that $e(x,C)\leq 2$ for all $x\in D$ by the
				maximality of $G'$ and Lemma \ref{ext}.
				
				First suppose that $H=W_5$. Then $G'=K_4$.
				Let $x_i$ ($1\leq i\leq 4$) be the four vertices of degree three in $H$
				and let $X=\{x_1,x_2,x_3,x_4\}$.
				If $e(x_i,C)\leq 1$ (resp. $e(x_i,C)=0$)
				for some $1\leq i\leq 4$, then (\ref{3dim}) implies that
				$e(v,D)\geq 3$ (resp. $e(v,D)\geq 4$) for
				at least three vertices $v\in C$, which gives $e(C,D)\geq 10$.
				It is impossible, since $e(x,C)\leq 2$ for all $x\in D$.
				Thus $e(x_i,C)=2$ for $1\leq i\leq 4$.
				A similar argument gives that $|N_G(X)\cap C|\geq 3$.
				Then $G[C\cup X]$ 
				is rigid by Lemma \ref{cycle}, contradicting the
				maximality of $G'$.
				
				Next suppose that $H=B_6$. 
				Then $4\leq |C|\leq 5$.
				Note that $e(v,D)\geq 2$ for each $v\in C$, since otherwise $\deg(v)+\deg(w)\leq |C|+6=|V|$ for some $w\in D$ with $uw\notin E$. 
				Hence we have $N_G(D)=C$. Let $V(H)=\{y_1,y_2,\ldots,y_6\}$, where $\deg_H(y_i)=3$ for $1\leq i \leq 2$ and $\deg_H(y_i)=4$ for $3\leq i \leq 6$. 
				If 
				$e(y_i,C)=2$ for $1\leq i \leq 6$,
				then $G[C\cup D]$ is rigid, by Lemma \ref{cycle}, a contradiction.
				So we may assume that $e(y_i,C)\leq 1$ for some $i$, $1\leq i \leq 6$, which implies that $\deg(y_i)\leq 5$. By (\ref{3dim}), there exist at least $|C|-1$ vertices in $C$, each of which has degree at least $|C|+2$. Then we have $e(C,D)=\sum_{v\in C}e(v,D)\geq (|C|+2-|C|+1)(|C|-1)+2=3|C|-1\geq 11$. Moreover, $e(D,C)=\sum_{v\in D}e(v,C)\leq 2(|D|-1)+1=11$. Thus we must have $|C|=4$, $e(y_i,C)= 1$ for some $i$, $3\leq i \leq 6$, and $e(y_j,C)=2$ for each $j\not= i$, $1\leq j \leq 6$. 
				Then $G[C\cup (D-\{y_i\})]$ is rigid by Lemma \ref{cycle}, a contradiction.
				
				Finally, suppose that $H$ is isomorphic to $C_7^1$ or $G_7^2$.
				Then $4\leq |C| \leq 6$. Let $V(H)=\{z_1,z_2,\ldots,z_7\}$. We have $\deg_H(z_i)=4$ for $1\leq i \leq 7$. For each vertex $v\in C$ we must have $e(v,D)\geq 3$, for otherwise $\deg(v)+\deg(w)\leq |C|+1+6=|V|$ for some $w\in D$ with $vw\notin E$. 
				Thus $N_G(D)=C$. If $e(z_i,C)\leq 1$ for some $i$, $1\leq i \leq 7$, then by (\ref{3dim}), there exist at least $|C|-1$ vertices in $C$ with degree at least $|C|+3$ in $G$. Hence  $e(C,D)=\sum_{v\in C}e(v,D)\geq (|C|+3-|C|+1)(|C|-1)+3=4|C|-1\geq 15$. However, $e(D,C)=\sum_{v\in D}e(v,C)\leq 2(|D|-1)+1=13$. It shows that $e(z_i,C)=2$ for
				$1\leq i\leq 7$. 
				Now we can use Lemma \ref{cycle} to deduce that $G[C\cup D]$ is rigid,
				a contradiction.
			\end{proof}

			By using that $q\geq 4$ and
			$e(v,C)\leq 2$ for all $v\in D$,
			we obtain that for all pairs $x,y$ of non-adjacent vertices of $H$
			$$\deg_H(x)+\deg_H(y)\geq \deg(x)-2+\deg(y)-2\geq n+1-4\geq |D|+1$$
			holds.
			Thus $H$ satisfies (\ref{3dim}), and hence,
			by the minimality of $G$, 
			$H$ is rigid, a contradiction.
		\end{proof}
		
		\section{On the rigidity of dense Erdős-Rényi random graphs}
		\label{sec:appl:random}
		
		In this section, we apply Theorem \ref{thm:delta:n/2+d} to settle a question raised by Peled and Peleg  on the rigidity of Erdős-Rényi random graphs. %
		This topic  has seen increasing interest in recent years, and one of its central
		 framing questions is: 
		\begin{quote}
		For which pairs $d=d(n)$ and $p=p(n)$ is the
		 random graph $G\sim G(n,p)$ asymptotically almost surely rigid in $\R^d$?
		\end{quote}
		This was first studied for fixed dimension $d=2$ \cite{JSS}. Subsequently, Lew, Nevo, Peled, and Raz \cite{LNPR} obtained an exact hitting time result for higher (but still fixed) dimensions $d\geq 1$, which was later extended to all $d(n)\leq c\log n/\log\log n$ by Krivelevich, Lew, and Michaeli \cite{KLM1}. 
		For $p = \omega\left(\log n/n \right)$, the following conjecture from \cite{KLM1} would provide an asymptotically tight bound on $d(n)$:
		\begin{conj} \cite[Conjecture 1.11]{KLM1}\label{conj:random}
			For every $p = \omega\left(\log n/n \right)$, the random graph $G \sim G(n,p)$ is a.a.s.\ $d$-rigid for $d=(1-o(1))(1-\sqrt{1-p})n$.
		\end{conj}
		Peled and Peleg verified this conjecture for $p=o(n^{-1/2})$ in a recent paper \cite{PP}.
		However, less is known in denser regimes.
		In particular, Conjecture \ref{conj:random} would imply that $G(n,1/2)$ is rigid for some $d(n)\sim \frac{2-\sqrt{2}}{2}n$.
		In this direction, Peled and Peleg asked whether there exists a constant $\varepsilon>0$ such that $G(n,1/2)$ is a.a.s.\ rigid in $\R^{\varepsilon n}$ (\cite[Section 4(3)]{PP}). In Theorem \ref{thm:random:1/2}, we answer their question affirmatively by showing that every $0<\varepsilon < 7/32=0.21875$ suffices. 
		Given that $\frac{2-\sqrt{2}}{2}\approx 0.2928$ is an upper bound on $\varepsilon$, this provides a reasonable estimate for the rigidity of $G(n,1/2)$. 
		We now recall and prove Theorem \ref{thm:random:1/2}.
		
		\random*
		
		\begin{proof}
			 For convenience, we assume that $n$ is even. Let $v_1,\dots, v_{n/2}, u_1,\dots u_{n/2}$ denote the vertices of $G$. Define $G_0=G$ and $G_i=G_{i-1}/\{u_i,v_i\}$ for $1\leq i\leq \frac n 2$. 
			Then $G_{\frac{n}{2}}$ is a graph with $\frac{n}{2}$ vertices in which the $i$th and $j$th vertices are adjacent if and only if at least one of the edges $\{u_iu_j, u_iv_j,v_iu_j, v_iv_j\}$ is present in $G$. 	Thus  $G_{\frac{n}{2}} \sim G(\frac n 2, 15/16)$.
			A.a.s., $G_{\frac{n}{2}}$ has minimum degree greater than $\frac{n}{4}+d$, see, e.g, \cite[Corollary 3.4]{BBrandom}. Hence, by Theorem \ref{thm:delta:n/2+d}, $G_{\frac{n}{2}}$ is $d$-rigid. %
			We claim that, a.a.s., for every $0\leq i\leq  \frac n 2 -1$, the graph $G_{i}$ can be obtained from $G_{i+1}$ by a $d$-dimensional  spider splitting operation and the addition of new edges. Since this operation preserves rigidity in $\R^d$, this implies that $G_i$ is rigid for every $0\leq i\leq \frac n 2$, thus completing the proof of the theorem. 
			
			To prove our claim, it remains to show that a.a.s.\ $|N_{G_i}(u_{i+1})\cap N_{G_i}(v_{i+1})|\geq d$ holds for every $0\leq i\leq \frac n 2 -1$. This will follow by straightforward probabilistic arguments.
			For $0\leq i\leq \frac n 2 -1$, let $X_i$ denote the size of $N_{G_i}(u_{i+1})\cap N_{G_i}(v_{i+1})$. The graph $G_i$ has $n-i$ vertices and $i$ of them represent contracted vertex pairs from $G_0$. Hence, for each $i$, we have $X_i=Y_i+Z_i$ where $Y_i\sim {\rm Bin}\big(i, \frac{9}{16}\big)$ and $Z_i\sim {\rm Bin}\big(n-2i-2,\frac{1}{4}\big)$ are independent variables. In particular,
            $$\E (X_i)= \frac{9i}{16}+\frac{n-2i-2}{4}\geq \frac{1}{4}(n-2)>d+\frac{n}{64}.$$
             Using Hoeffding's inequality, we obtain that $\Prob(X_i<d)\leq e^{-\Omega(n)}$ for every $i$. Hence, by the union bound, the probability that there is some $0\leq i\leq \frac n 2 -1$ with $X_i< d$ is $o(1)$.
		\end{proof}
		
		Prior to this result, the best known lower bound for $d=d(n)$ was $d = \Omega(n/\log n)$, which followed as a special case of \cite[Theorem 1.10]{KLM1}. 
		With straightforward modifications, the proof of Theorem \ref{thm:random:1/2} yields the following more general result.
		
		\begin{theorem}\label{thm:random:any:p}
			For every constant $p\in \R$ with $0< p<1$, there exists $\varepsilon >0$ such that  $G\sim G(n,p)$ is a.a.s.\ $(\varepsilon n)$-rigid.
		\end{theorem}

		\section{Concluding remarks}	\label{sec:conluding}

        \subsection{Edge count for degree sum conditions}
        It is clear that $\delta(G)\geq \frac{n+d}{2}-1$ implies
		$|E|\geq \frac{n(n+d-2)}{4}$. 
		We show that that the degree sum condition $\deg(u)+\deg(v)\geq n+d-2$
		gives rise to the same bound.
		
		\begin{lemma}
			\label{ecount}
			Let $G=(V,E)$ be a graph.
			If
			\begin{equation}
				\label{eq2}
				\deg(u)+\deg(v)\geq n+d-2
			\end{equation}
			for all pairs of non-adjacent vertices in $G$, then
			$|E|\geq \frac{n(n+d-2)}{4}$.
		\end{lemma}
		
		\begin{proof}
			We shall use the following well-known inequality, which holds for
			any $n$-tuple of real numbers $a_1,a_2,\dots,a_n$:
			\begin{equation}
				\label{ineq}
				\frac{(\sum_{i=1}^n a_i)^2}{n}\leq \sum_{i=1}^n a_i^2
			\end{equation}
			First we consider the non-edges of $G$ and use (\ref{eq2}) to deduce that
			$$\frac{n(n-1)-2|E|}{2}(n+d-2)\leq \sum_{u,v\in V, uv\notin E} \deg(u)+\deg(v)=$$
			$$=\sum_{v\in V} (n-1-\deg(v)) \deg(v)= 2|E|(n-1) - \sum_{v\in V} \deg(v)^2$$
			holds. By applying (\ref{ineq}), we can upper bound the RHS by
			$$2|E|(n-1)-\frac{(\sum_v \deg(v))^2}{n}=\sum_{v\in V}(n-1-\frac{2|E|}{n})\deg(v)=2|E|(n-1-\frac{2|E|}{n})=2|E|(\frac{n(n-1)-2|E|}{n}).$$
			By comparing the two sides we get
			$|E|\geq \frac{n(n+d-2)}{4}$, as claimed.
		\end{proof}

        \subsection{Conditions for global rigidity}
		
			A framework $(G,p)$ in $\R^d$ is {\it globally rigid} if every $d$-dimensional framework $(G,q)$ in which corresponding edge lengths are the same as in $(G,p)$, is congruent to $(G,p)$. As in the case of rigidity,
		the global rigidity of a generic framework is determined by its graph, for every fixed $d\geq 1$.
		Thus we call $G$ {\it globally $d$-rigid} if every (or equivalently, if some) generic $d$-dimensional framework $(G,p)$ 
        is globally rigid.
        We refer the reader to \cite{JW} for more details concerning
	globally rigid graphs and frameworks.
		
        The following result shows that
		whenever a minimum degree or degree sum condition implies rigidity 
		in $\R^{d}$, it also
		implies global rigidity in $\R^{d-1}$.
		
		\begin{theorem}
			\label{theorem:dimensiondropping}
			\cite{Jext}
			Let $G$ be a graph and $d \geq 2$. If $G$ is rigid in $\R^{d}$, then $G$ is globally rigid 
            in $\R^{d-1}$.
		\end{theorem}
		
		It is well-known that every globally $d$-rigid graph on at least $d+2$ vertices
		is $(d+1)$-connected. Hence, combining Theorem \ref{theorem:dimensiondropping} with our results leads to
		tight bounds concerning global $d$-rigidity.
		In particular, we obtain the following corollaries.
		
		\begin{cor}
			\label{ddimGGR}
			Let $G=(V,E)$ be a graph on $n$ vertices and let $d\geq 2$. If either\\
             (i) \(\delta(G) \geq \left\lceil\frac{n+d-2}{2}\right\rceil\) and \(n \geq 29d\), or\\
             (ii) \(\eta(G) \geq n+d-2\) and \(n \geq d(d+2)\),\\
            then \(G\) is globally rigid in \(\mathbb{R}^{d-1}\).
		\end{cor}

        It was proved in \cite{JJchapter} that if $\delta(G)\geq
		\frac{n+1}{2}$, then
		$G$ is globally rigid in $\R^2$. This statement is a special case of the following corollary. 
        A graph $G$ is called {\it redundantly $d$-rigid} if $G-e$ is $d$-rigid
		      for all $e\in E(G)$.
		
		\begin{cor}
			\label{2dimGGR}
			Let $G=(V,E)$ be a graph on $n\geq 5$ vertices. If $\eta(G)\geq n+1$,
            then
			$G$ is globally rigid in $\R^2$.
		\end{cor}

		\begin{proof}
			We are done by Theorems \ref{rigid graph in R^3}
			and \ref{theorem:dimensiondropping}, unless $G$ is one of the
			four exceptions. 

              The four exceptions are all
			$3$-connected and redundantly 2-rigid, so they are
			also
			globally 2-rigid by \cite{JJconnrig}.
		\end{proof}

		\section*{Acknowledgements}
		
		The first author was supported by the
		MTA-ELTE Momentum Matroid Optimization Research Group and the National Research, Development and Innovation Fund of Hungary, financed under the ELTE TKP 2021‐NKTA‐62 funding scheme.
		The second author was supported by the  China Scholarship Council No. 202406290008 and the China Postdoctoral Science Foundation No. 2024M754212.

	\end{document}